\documentclass[11pt]{amsart}

\usepackage[margin=1.1in]{geometry}
\usepackage{amsmath,amssymb,amsthm,mathtools}
\usepackage[T1]{fontenc}
\usepackage[utf8]{inputenc}
\usepackage[english]{babel}
\usepackage[expansion=false]{microtype}
\usepackage{xcolor}
\usepackage[colorlinks=true,linkcolor=blue,citecolor=blue,urlcolor=blue]{hyperref}

\newtheorem{theorem}{Theorem}[section]
\newtheorem{proposition}[theorem]{Proposition}
\newtheorem{corollary}[theorem]{Corollary}
\newtheorem{lemma}[theorem]{Lemma}
\newtheorem{question}[theorem]{Question}
\theoremstyle{definition}
\newtheorem{definition}[theorem]{Definition}

\theoremstyle{remark}
\newtheorem{remark}[theorem]{Remark}

\newcommand{\Hh}{\mathbb H}
\newcommand{\R}{\mathbb R}
\newcommand{\Z}{\mathbb Z}
\newcommand{\T}{\mathbb T}
\newcommand{\Sph}{\mathbb S}
\newcommand{\Aut}{\operatorname{Aut}}
\newcommand{\Out}{\operatorname{Out}}

\newcommand{\Fix}{\operatorname{Fix}}

\newcommand{\tr}{\operatorname{tr}}
\newcommand{\Hom}{\operatorname{Hom}}
\newcommand{\SU}{\operatorname{SU}}
\newcommand{\Sp}{\operatorname{Sp}}
\newcommand{\Jac}{\operatorname{Jac}}
\newcommand{\id}{\operatorname{id}}
\newcommand{\ab}{\operatorname{ab}}
\newcommand{\rhoexp}{\operatorname{R}}

\title{Quaternionic Extensions of Hyperbolic Toral Automorphisms}
\dedicatory{To the memory of Jacob Palis (1940--2025)}
\author{Alberto Verjovsky}
\address{Instituto de Matem\'aticas, \\ Unidad  Cuernavaca \\ Universidad Nacional Aut\'onoma de M\'exico\\
Apartado Postal 273, Administraci\'on de Correos No.~3\\
C.P. 62251 Cuernavaca, Morelos, M\'exico\\
\texttt{alberto\@matcuer.unam.mx}}
\email{}
\date{July 30, 2026}


\begin{document}
\begin{abstract}
We study real-analytic word maps of
\(S^3\times S^3\cong \SU(2)\times\SU(2)\) obtained by evaluating
Nielsen automorphisms of the free group \(F_2\).  Fried studied these word
maps and computed their topological entropy.  For each maximal torus
\(T\subset\SU(2)\), the product \(T\times T\) is invariant and carries the
corresponding toral dynamics.  Their union is the commuting locus.  We
describe its regular part as a twisted conjugacy-sphere bundle over the
punctured pillowcase and determine the local topology at the four central
singularities.  Periodic toral points give periodic conjugacy two-spheres,
which obstruct Anosov hyperbolicity on \(S^3\times S^3\).  Normalized images
of the coordinate three-cycles converge to stable and unstable
Ruelle--Sullivan eigencurrents supported on the commuting locus.

The construction extends to \(\Sp(1)^n\).  The higher-rank commuting locus
has dimension \(n+2\), periodic conjugacy two-spheres persist, and the action
on \(H_{3k}((S^3)^n)\) is \(\Lambda^k A\).  For \(n\geq3\), IA-automorphisms
can act nontrivially on the \(\SU(2)\)-character variety although they are
trivial on abelianization.  The degree-three eigencurrents again have a
Ruelle--Sullivan description; higher exterior-power classes lead to a
compactness problem for normal currents.  We also relate the nonuniqueness of
Nielsen lifts to Fried's computation of the mapping class group of
\(S^3\times S^3\), where the group \(\Theta_7\) of homotopy seven-spheres
appears in the kernel of the action on homology.
\end{abstract}

\maketitle

\section{Introduction}

Let
\[
 A=\begin{pmatrix}a&b\\ c&d\end{pmatrix}\in \mathrm{SL}(2,\Z)
\]
be hyperbolic, so that \(|\tr A|>2\).  The induced automorphism of the
real two-torus is one of the basic examples of an Anosov diffeomorphism.
We consider a noncommutative extension of this construction to
\[
 \Sph^3\times\Sph^3\cong \SU(2)\times\SU(2),
\]
where \(\Sph^3\) is regarded as the multiplicative group of unit
quaternions.

Consider first the formula
\begin{equation}\label{eq:power-product}
 M_A(q_1,q_2)=\bigl(q_1^a q_2^b,q_1^c q_2^d\bigr).
\end{equation}
On every commuting two-torus this reduces to the usual map induced by
\(A\).  In contrast with the abelian case, however, the assignment
\(A\mapsto M_A\) is not a group homomorphism: generally
\[
 M_A\circ M_B\ne M_{AB}.
\]
Thus a factorization of a matrix into elementary matrices cannot be
reassembled by simply collecting quaternionic exponents.

The free group
\(F_2=\langle x,y\rangle\) provides the corresponding noncommutative
construction.  An automorphism of \(F_2\) can be evaluated
on any pair of elements of a group, and its inverse word gives the inverse
map.  The classical surjection
\[
 \Aut(F_2)\longrightarrow \mathrm{GL}(2,\Z)
\]
is obtained by abelianization.  Hence every integral unimodular matrix
has real-analytic diffeomorphic lifts to \(\SU(2)^2\).  The lift is
determined by the full noncommutative word, not by the exponent matrix
alone.

This construction belongs to the theory of \emph{word maps} studied by
Fried in \cite{Fried}.  For invertible word maps on \(S^3\times S^3\),
Fried computed the exact topological entropy, studied the quotient by
simultaneous conjugation and its Fricke level sets, and determined the
isotopy classes represented by word maps.  His mapping-class calculation
also exhibits the role of exotic $7$-spheres in the failure of the natural
map from isotopy classes to \(\mathrm{GL}(2,\mathbb Z)\) to split.  The commuting locus in $S^3\times S^3$ has an explicit singular geometry, and simultaneous conjugation produces periodic conjugacy spheres.  Normalized iterates of coordinate three-cycles give stable and unstable currents.  The dependence on the chosen free-group automorphism is retained throughout.

The resulting maps preserve Haar volume.  Moreover, for every maximal
torus \(T\subset\SU(2)\), the product \(T\times T\cong\mathbb T^2\) is
invariant, and the restriction to it is the hyperbolic toral automorphism
determined by \(A\).  The union of all these invariant two-tori is exactly
the commuting locus
\[
 \mathcal C=\{(q_1,q_2)\in\SU(2)^2:q_1q_2=q_2q_1\}.
\]
Distinct maximal tori of \(\SU(2)\) intersect in the center
\(\{\pm1\}\); consequently, the corresponding products \(T\times T\)
are disjoint away from the four central pairs \((\pm1,\pm1)\).  The ambient map is
nevertheless not Anosov on the six-manifold: simultaneous conjugation turns a generic toral periodic
point into an entire conjugacy sphere which is fixed pointwise by the
appropriate return map.  Passing to the
\(\SU(2)\)-character variety removes these symmetry directions.  For the
Fibonacci cat map the quotient is an explicit polynomial automorphism of
a character three-ball.  It preserves the Fricke commutator invariant, so
the full quotient is not globally hyperbolic.  Its boundary is the
pillowcase of conjugacy classes of commuting pairs, and the boundary
dynamics is the pillowcase quotient of the classical cat map.

Homology and de Rham currents carry an expanding and contracting action.  Let $\mu_u$ be the unstable eigenvalue of $A$, with $|\mu_u|>1$.
Then suitably normalized forward images of the coordinate three-cycles
converge to an unstable eigen-current $\mathcal T^u$, while normalized
backward images converge to a stable eigen-current $\mathcal T^s$.  They satisfy
\[
 (F_A)_\#\mathcal T^u=\mu_u\mathcal T^u,
 \qquad
 (F_A)_\#\mathcal T^s=\mu_u^{-1}\mathcal T^s.
\]
Thus their rays are invariant, and their magnitudes grow or decay
exponentially at rate $|\mu_u|$.

The rank-two situation has a special feature that disappears in higher rank.
For $F_2$ the kernel of abelianization is the group of inner automorphisms,
so the matrix $A$ determines the induced map on the $\SU(2)$-character
variety.  For $F_n$, $n\geq3$, this is no longer true.  For $n\geq3$, the same construction on $\Sp(1)^n$ gives a commuting locus of dimension $n+2$ and the action $\Lambda^kA$ on $H_{3k}$.  An explicit IA-automorphism acts trivially on abelianization but nontrivially on the character variety.  Under a total irrationality hypothesis, the degree-three eigencurrent construction persists; higher degrees lead to a normal-current compactness problem.
Related higher-rank word maps and questions about metric entropy,
ergodicity, and mapping classes were considered independently by A. Gogolev
and F. Rodr\'iguez Hertz in unpublished notes communicated to the author in
August 2026.  Their rank-three example is discussed in
Section~\ref{sec:higher-rank}.

\section{Unit quaternions and word maps}

Let
\[
 \Hh=\{x_0+x_1\mathbf i+x_2\mathbf j+x_3\mathbf k:x_r\in\R\}
\]
be the quaternion division algebra, with
\(\mathbf i^2=\mathbf j^2=\mathbf k^2=\mathbf i\mathbf j\mathbf k=-1\).
Its unit sphere
\[
 \Sph^3=\{q\in\Hh:|q|=1\}
\]
is a compact Lie group, naturally isomorphic to \(\SU(2)\).

Let \(F_2=\langle x,y\rangle\).  An \emph{elementary Nielsen
transformation} is an automorphism obtained by one of the following
operations on the free basis: interchanging the two generators,
replacing one generator by its inverse, or multiplying one generator on
the left or on the right by the other.  These elementary transformations
generate \(\Aut(F_2)\).  An arbitrary element of
\(\Aut(F_2)\), equivalently a finite composition of elementary Nielsen
transformations, will be called a \emph{Nielsen automorphism}.

For a reduced word \(w\in F_2\) and a pair
\((q_1,q_2)\in \SU(2)^2\), write \(w(q_1,q_2)\) for the element obtained
by substituting \(q_1\) for \(x\) and \(q_2\) for \(y\).

\begin{definition}
For a Nielsen automorphism \(\Phi\in\Aut(F_2)\), with
\[
 \Phi(x)=w_1(x,y),\qquad \Phi(y)=w_2(x,y),
\]
define the associated quaternionic Nielsen map by
\[
 F_\Phi(q_1,q_2)=\bigl(w_1(q_1,q_2),w_2(q_1,q_2)\bigr).
\]
\end{definition}
These maps are the invertible \emph{word maps} of Fried \cite{Fried},
specialized here to the compact Lie group \(\SU(2)\).

Our convention gives
\begin{equation}\label{eq:composition}
 F_\Phi\circ F_\Psi=F_{\Psi\circ\Phi}.
\end{equation}
Indeed, evaluating \(\Phi(x)\) and \(\Phi(y)\) on
\((\Psi(x),\Psi(y))\) is the evaluation of \(\Psi(\Phi(x))\) and
\(\Psi(\Phi(y))\).  One may reverse the convention to obtain a left
rather than a right action; no dynamical statement below depends on this
choice.

\begin{proposition}\label{prop:diffeomorphism}
For every \(\Phi\in\Aut(F_2)\), the map \(F_\Phi\) is a real-analytic
diffeomorphism of \(\SU(2)^2\), and
\[
 F_\Phi^{-1}=F_{\Phi^{-1}}.
\]
\end{proposition}

\begin{proof}
Word evaluation is real analytic because multiplication and inversion in
a Lie group are real analytic.  Formula \eqref{eq:composition} gives
\[
 F_\Phi\circ F_{\Phi^{-1}}=F_{\id}=\id,
 \qquad
 F_{\Phi^{-1}}\circ F_\Phi=\id.
\]
\end{proof}

\section{Elementary Nielsen transformations and matrix lifts}

In the terminology introduced above, convenient elementary Nielsen
transformations are
\[
 (x,y)\mapsto(xy,y),\qquad
 (x,y)\mapsto(x,xy),
\]
\[
 (x,y)\mapsto(x^{-1},y),\qquad
 (x,y)\mapsto(y,x).
\]
Their evaluations are
\begin{align*}
 N_{12}(q_1,q_2)&=(q_1q_2,q_2),
 &N_{12}^{-1}(u,v)&=(uv^{-1},v),\\
 N_{21}(q_1,q_2)&=(q_1,q_1q_2),
 &N_{21}^{-1}(u,v)&=(u,u^{-1}v),\\
 I_1(q_1,q_2)&=(q_1^{-1},q_2),
 &I_1^{-1}&=I_1,\\
 P(q_1,q_2)&=(q_2,q_1),
 &P^{-1}&=P.
\end{align*}

Abelianization sends \(F_2\) to
\(F_2^{\mathrm{ab}}=F_2/[F_2,F_2]\cong\Z^2\).  With the standard
column convention, in which the \(j\)-th column records the exponent-sum
vector of the image of the \(j\)-th free generator, it induces a
surjective homomorphism
\begin{equation}\label{eq:abmap}
 \ab:\Aut(F_2)\longrightarrow\mathrm{GL}(2,\Z).
\end{equation}
For the dynamics it is more convenient to use the transposed
\emph{row exponent matrix}
\begin{equation}\label{eq:rowmatrix}
 \rhoexp(\Phi):=\ab(\Phi)^t.
\end{equation}
Thus the first row of \(\rhoexp(\Phi)\) contains the exponent sums of
\(\Phi(x)\), and the second row those of \(\Phi(y)\).  Because of the
transpose, \(\rhoexp\) is an anti-homomorphism:
\[
 \rhoexp(\Phi\circ\Psi)=\rhoexp(\Psi)\rhoexp(\Phi).
\]
This convention is exactly the one compatible with the contravariant
composition rule \eqref{eq:composition} and with the action on angular
coordinates below.

The kernel of either \(\ab\) or \(\rhoexp\) is
\[
 IA_2:=\ker(\ab)=\ker(\rhoexp),
\]
the group of \emph{IA-automorphisms}; the letters IA mean ``identity on
abelianization.''  An \emph{inner automorphism} is a conjugation
\(\iota_w(\gamma)=w\gamma w^{-1}\).  A classical theorem special to rank two states that
\[
 IA_2=\operatorname{Inn}(F_2).
\]
See, for example, \cite{Nielsen1917,Nielsen1924,MagnusKarrassSolitar}.
Consequently,
\[
 \Out(F_2):=\Aut(F_2)/\operatorname{Inn}(F_2)
 \cong\mathrm{GL}(2,\Z).
\]

\begin{definition}
A \emph{quaternionic lift} of \(A\in\mathrm{GL}(2,\Z)\) is a map
\(F_\Phi\) for which \(\rhoexp(\Phi)=A\).  Equivalently, the standard
abelianization matrix of \(\Phi\) is \(A^t\).
\end{definition}

\begin{theorem}\label{thm:existence}
Every matrix \(A\in\mathrm{GL}(2,\Z)\) admits a quaternionic lift
\[
 F_\Phi:\Sph^3\times\Sph^3\longrightarrow\Sph^3\times\Sph^3
\]
which is a real-analytic diffeomorphism.  Such a lift can be constructed
by factoring \(A\) into elementary integral matrices and replacing the
factors by the corresponding Nielsen transformations.
\end{theorem}

\begin{proof}
Surjectivity of \eqref{eq:abmap} gives \(\Phi\in\Aut(F_2)\) with
\(\ab(\Phi)=A^t\), equivalently \(\rhoexp(\Phi)=A\).
Proposition~\ref{prop:diffeomorphism} proves that its evaluation is a
real-analytic diffeomorphism.  Equivalently, one may factor \(A\) into
elementary matrices and use the corresponding Nielsen transformations,
with the order reversed as dictated by the anti-homomorphism
\(\rhoexp\).
\end{proof}

\begin{proposition}[Nonuniqueness of the lift]\label{prop:nonunique}
A matrix \(A\in\mathrm{GL}(2,\Z)\) canonically determines an outer
automorphism class \([\Phi_A]\in\Out(F_2)\), but not a unique element of
\(\Aut(F_2)\).  In particular, different factorizations of \(A\) into
elementary matrices may produce different Nielsen lifts.  If
\(\Phi,\Psi\in\Aut(F_2)\) have the same row exponent matrix, then there exists
\(w\in F_2\) such that
\[
 \Psi=\iota_w\circ\Phi.
\]
After evaluation on \(\SU(2)^2\), the two lifts differ by simultaneous
conjugation with the point-dependent element \(w(q_1,q_2)\).  Hence they
induce the same map on the character variety
\[
 \SU(2)^2/\SU(2),
\]
although they need not coincide on \(\SU(2)^2\).
\end{proposition}

\begin{proof}
The equality \(\rhoexp(\Phi)=\rhoexp(\Psi)\) is equivalent to
\(\ab(\Phi)=\ab(\Psi)\), so the automorphism
\(\Psi\Phi^{-1}\) belongs to \(IA_2\).  Since
\(IA_2=\operatorname{Inn}(F_2)\), it equals \(\iota_w\) for some
\(w\in F_2\).  Evaluating this identity on \(\SU(2)^2\), one finds
that the two lifts differ by simultaneous conjugation by the
point-dependent element \(w(q_1,q_2)\).  Consequently, they determine
the same map on the quotient
\[
 (\SU(2)\times\SU(2))/\SU(2),
\]
and hence induce the same transformation of the \(\SU(2)\)-character
variety.
\end{proof}

\subsection{Inner Nielsen maps and the mapping class group}
\label{subsec:exotic-mapping-classes}

The ambiguity in Proposition~\ref{prop:nonunique} is visible in the smooth
mapping class group.  Let
\[
 \pi=\pi_0\operatorname{Diff}(S^3\times S^3),
 \qquad
 \kappa:\pi\longrightarrow \mathrm{GL}(2,\Z)
\]
be the action on \(H_3(S^3\times S^3;\Z)\).  With the convention of Section~2, word evaluation gives an anti-homomorphism
\[
 \Aut(F_2)\longrightarrow \pi;
\]
equivalently, it is a homomorphism from the opposite group.  Fried proved that
this map is onto and that \(\kappa\) is onto but does not split \cite{Fried}.  Thus
every matrix has individual word-map representatives, but their isotopy
classes cannot be chosen to form a copy of \(\mathrm{GL}(2,\Z)\) in \(\pi\).

The kernel of \(\kappa\) already contains nontrivial inner Nielsen maps.  With
our commutator convention
\[
 [x,y]=xyx^{-1}y^{-1},\qquad c=[x,y],
\]
let \(\iota_c\) denote conjugation by \(c\).  Its evaluation is
\begin{equation}\label{eq:exotic-inner-map}
 E_c(q_1,q_2)=
 \bigl(Cq_1C^{-1},Cq_2C^{-1}\bigr),
 \qquad C=[q_1,q_2].
\end{equation}
Since \(\iota_c\) is inner, \(E_c\) acts trivially on
\(H_3(S^3\times S^3;\Z)\).

\begin{proposition}[Fried]\label{prop:exotic-nielsen}
The diffeomorphism \(E_c\) is homotopic to the identity but is not smoothly
isotopic to the identity.  Up to the sign coming from the commutator
convention, its mapping class is a generator of the subgroup
\[
 \Theta_7\cong\Z/28
\]
of homotopy seven-spheres contained in \(\ker\kappa\).
\end{proposition}

\begin{proof}
\emph{Fried's construction, quoted.}  Fried associates to an inner word map
a class in the subgroup of mapping classes acting trivially on homology.
For the basic commutator, he represents this class by a compactly supported
commutator diffeomorphism of \(\mathbb R^6\).  Gluing two seven-discs by the
induced diffeomorphism of \(S^6\) gives the corresponding homotopy
seven-sphere.  Milnor's bilinear pairing and the Eells--Kuiper invariant
show that this element generates \(\Theta_7\), up to sign; see
\cite[\S2]{Fried}.  Hence the word map is not isotopic to the identity.

\emph{Locating $E_c$, argued here.}  Fried's homomorphism from the
inner-word-map subgroup to $H_3(S^3\times S^3;\Z)$ records the
abelianization of the conjugating word.  It vanishes for $c=[x,y]$, so the
class of $E_c$ lies in the central $\Theta_7$ subgroup identified above,
and $E_c$ is homotopic to the identity.
\end{proof}

Fried describes \(K=\ker\kappa\) as a finite central quotient of the integral
Heisenberg group.  In his notation,
\[
 K\cong H_{\mathbb Z}/\langle z^{28}\rangle,
\]
and the central element represented by \(z\) is the class above.  The
non-splitting of \(\kappa\) uses the fact that this central element has order
\(28\) \cite{Fried}.

Kreck's exact sequences for highly connected almost-parallelizable manifolds
provide the general framework \cite{Kreck1979}.  For orientation-preserving
diffeomorphisms of $S^3\times S^3$, Krylov gave a presentation of the mapping
class group and related it to the integral Jacobi group \cite{Krylov2003}.

With
\[
 e_1=[S^3\times\{1\}],\qquad e_2=[\{1\}\times S^3],
\]
the intersection form on $H_3(S^3\times S^3;\mathbb Z)$ is represented by
\[
 J=\begin{pmatrix}0&1\\-1&0\end{pmatrix}.
\]
Hence the action of an orientation-preserving diffeomorphism on middle
homology belongs to
\[
 \operatorname{Sp}_2(\mathbb Z)
 =\{M\in\operatorname{GL}(2,\mathbb Z):M^tJM=J\}
 =\operatorname{SL}(2,\mathbb Z).
\]
Put
\[
 \pi^+=\pi_0\operatorname{Diff}^+(S^3\times S^3).
\]
Krylov starts from the exact sequence
\begin{equation}\label{eq:krylov-split-quotient}
 1\longrightarrow\mathbb Z^2
 \longrightarrow \pi^+/\Theta_7
 \longrightarrow \operatorname{SL}(2,\mathbb Z)
 \longrightarrow1
\end{equation}
and proves that it splits.

The group denoted by $\Gamma_J$ in Krylov's paper is the integral Jacobi
group
\begin{equation}\label{eq:jacobi-definition}
 \Gamma_J=\operatorname{SL}(2,\mathbb Z)\ltimes\mathbb Z^2.
\end{equation}
Krylov writes elements as pairs $(M,X)$, where $M\in\operatorname{SL}(2,\mathbb Z)$
and $X\in\mathbb Z^2$ is a row vector, with multiplication
\begin{equation}\label{eq:jacobi-law}
 (M,X)(M',X')=(MM',XM'+X').
\end{equation}
Thus $\mathbb Z^2$ is a normal abelian subgroup and the quotient is
$\operatorname{SL}(2,\mathbb Z)$.  The splitting of
\eqref{eq:krylov-split-quotient} gives
\begin{equation}\label{eq:krylov-jacobi-quotient}
 \pi^+/\Theta_7\cong\Gamma_J.
\end{equation}
Equivalently, the full orientation-preserving mapping class group fits into
the central extension
\begin{equation}\label{eq:krylov-jacobi}
 1\longrightarrow \Theta_7
 \longrightarrow \pi^+
 \longrightarrow \Gamma_J
 \longrightarrow1,
 \qquad \Theta_7\cong\mathbb Z/28.
\end{equation}

For reference, Krylov uses the following presentation of $\Gamma_J$:
\begin{equation}\label{eq:jacobi-presentation}
\begin{split}
 \Gamma_J=\langle y,u,a,b\mid{}&yuy=uyu,\quad (yuy)^4=1,\quad ab=ba,\\
 &ay=yab,\quad au=ua,\quad by=yb,\quad bu=uba^{-1}\rangle.
\end{split}
\end{equation}
Here $y,u$ project to generators of $\operatorname{SL}(2,\mathbb Z)$ and
$a,b$ generate the normal subgroup $\mathbb Z^2$.  In Krylov's matrix
convention,
\[
 y=\begin{pmatrix}1&1\\0&1\end{pmatrix},\qquad
 u=\begin{pmatrix}1&0\\-1&1\end{pmatrix}.
\]
The relations involving $a,b$ in \eqref{eq:jacobi-presentation} are precisely
the conjugation action of these two matrices on $\mathbb Z^2$.

The occurrence of a Heisenberg group belongs to the kernel before quotienting
by $\Theta_7$.  Let $H_{\mathbb Z}$ be the integral Heisenberg group
\[
 H_{\mathbb Z}=\langle a,b,z\mid z\ \text{central},\ [a,b]=z\rangle,
\]
and put
\[
 H_{28}=H_{\mathbb Z}/\langle z^{28}\rangle.
\]
Fried's calculation gives the kernel of the action on $H_3$ as $H_{28}$.
Accordingly Krylov has the exact sequence
\begin{equation}\label{eq:krylov-heisenberg}
 1\longrightarrow H_{28}\longrightarrow\pi^+
 \longrightarrow\operatorname{SL}(2,\mathbb Z)\longrightarrow1.
\end{equation}
After quotienting by the central subgroup generated by $z$, the two
Heisenberg generators commute and give the $\mathbb Z^2$ subgroup in
$\Gamma_J$.  Thus the two descriptions are compatible:
\[
 H_{28}/\langle z\rangle\cong\mathbb Z^2,
 \qquad
 \pi^+/\Theta_7\cong\operatorname{SL}(2,\mathbb Z)\ltimes\mathbb Z^2.
\]
Krylov also computes the cohomology of $\Gamma_J$ and determines the
$2$-cocycle corresponding to the central extension
\eqref{eq:krylov-jacobi}.

\subsubsection*{Quaternionic word representatives}
The $\operatorname{SL}(2,\mathbb Z)$ part has simple quaternionic word
representatives.  Put
\[
 Y(s,t)=(s,st),\qquad
 U(s,t)=(t^{-1}s,t).
\]
Their inverses are
\[
 Y^{-1}(s,t)=(s,s^{-1}t),\qquad
 U^{-1}(s,t)=(ts,t).
\]
With respect to the basis $(e_1,e_2)$ and the row convention used in this
paper, their actions on $H_3$ are
\[
 Y_*=\begin{pmatrix}1&0\\1&1\end{pmatrix},
 \qquad
 U_*=\begin{pmatrix}1&-1\\0&1\end{pmatrix}.
\]
These are the transposes of Krylov's displayed matrices because of the row
exponent convention used here.  Direct quaternion multiplication gives
\begin{equation}\label{eq:YUY}
 YUY=UYU;
\end{equation}
both sides send $(s,t)$ to
\[
 (t^{-1},t^{-1}st).
\]

Two homologically trivial word maps occurring in this description are
\[
 \mathcal A(s,t)=(tst^{-1},t),\qquad
 \mathcal B(s,t)=(s,sts^{-1}).
\]
They induce the identity on $H_3$.  Their classes modulo $\Theta_7$ give the
two generators of the normal $\mathbb Z^2$ subgroup of $\Gamma_J$.  Before
passing to this quotient their commutator is measured by the central
$\Theta_7$ coordinate in $H_{28}$.  Fried's commutator-conjugation map in
Proposition~\ref{prop:exotic-nielsen} represents a generator of this central
subgroup, up to the chosen commutator convention.

\subsection{The factor set of a choice of Nielsen lifts}
\label{subsec:factor-set}

Choose, set-theoretically, a Nielsen lift \(F_A\) for every
\(A\in\mathrm{GL}(2,\Z)\), and put \(s(A)=[F_A]\in\pi\).  Its action on
\(H_3(S^3\times S^3;\mathbb Z)\) is \(A\), as follows directly by restricting
to the two coordinate three-spheres (see Section~\ref{sec:eigencurrent}).  Then
\begin{equation}\label{eq:factor-set}
 c(A,B)=s(A)s(B)s(AB)^{-1}\in K.
\end{equation}
Associativity gives
\begin{equation}\label{eq:factor-set-identity}
 c(A,B)c(AB,C)
 =s(A)c(B,C)s(A)^{-1}c(A,BC).
\end{equation}
Changing the section changes \(c\) by the corresponding nonabelian
coboundary.  Thus \(\kappa\) splits if and only if the section can be chosen
with \(c(A,B)=1\) for all \(A,B\).  Fried's theorem says that no such section exists.

Crowley's calculation for connected sums of \(S^p\times S^p\), \(p=3,7\),
contains the same failure of realization.  In the case
\(S^3\times S^3\), his Corollary~1.4 states that the symplectic action on
middle homology has no section even at the level of isotopy classes
\cite{Crowley2011}.  For connected sums with more than one summand he proves
a stronger obstruction already at the level of homotopy classes.

Gogolev and Rodr\'iguez Hertz considered the same realization problem in
unpublished notes, starting from explicit monomial representatives and
commutator corrections.  Their notes contain related calculations on mapping classes and realization.  A preliminary group presentation appearing in those notes is not used here.

\section{Power-product maps}

Given
\(A=\left(\begin{smallmatrix}a&b\\c&d\end{smallmatrix}\right)\), the
formula \eqref{eq:power-product} is always a real-analytic self-map.  At
the identity it satisfies
\begin{equation}\label{eq:derivativeidentity}
 D M_A(1,1)=A\otimes I_3,
\end{equation}
under the identification
\(T_1\SU(2)\cong\operatorname{Im}\Hh\cong\R^3\).  Thus, when
\(A\in\mathrm{GL}(2,\Z)\), it is a local diffeomorphism near \((1,1)\).

\subsection{An example and the global invertibility problem}
\label{subsec:power-product-example}

Some nontrivial examples can be handled exactly.  For instance, let
\[
 A_0=\begin{pmatrix}3&1\\2&1\end{pmatrix},
 \qquad
 M_{A_0}(q_1,q_2)=(q_1^3q_2,q_1^2q_2).
\]
If
\[
 (u,v)=(q_1^3q_2,q_1^2q_2),
\]
then
\[
 uv^{-1}=q_1,
 \qquad
 q_2=(uv^{-1})^{-2}v.
\]
Hence
\[
 M_{A_0}^{-1}(u,v)
 =\bigl(uv^{-1},(uv^{-1})^{-2}v\bigr),
\]
so $M_{A_0}$ is a real-analytic diffeomorphism.  This diffeomorphism does not itself arise from a single elementary Nielsen generator.

\begin{proposition}\label{prop:powerproductcriterion}
If the two words
\[
 x^a y^b,\qquad x^c y^d
\]
form a free basis of \(F_2\), then the power-product map \(M_A\) is a
real-analytic diffeomorphism of \(\SU(2)^2\).
\end{proposition}

\begin{proof}
A bijective correspondence between one free basis of \(F_2\) and another
extends uniquely to an automorphism of \(F_2\); see, for example,
\cite{MagnusKarrassSolitar}.  Under the hypothesis, this applies to the
assignment \(x\mapsto x^a y^b\), \(y\mapsto x^c y^d\), which therefore
defines an automorphism \(\Phi\in\Aut(F_2)\).  Then \(M_A=F_\Phi\), and
Proposition~\ref{prop:diffeomorphism} applies.
\end{proof}

The converse at the level of evaluation on \(\SU(2)\) is not addressed by
this argument: a word endomorphism could in principle evaluate bijectively
on this particular compact group without being an automorphism of \(F_2\).

\begin{remark}\label{rem:powerproductclassification}
No general criterion is established here for deciding when the
power-product map
\[
 M_A(q_1,q_2)=(q_1^a q_2^b,q_1^c q_2^d),
 \qquad
 A=\begin{pmatrix}a&b\\c&d\end{pmatrix}\in\mathrm{GL}(2,\Z),
\]
is a global diffeomorphism of \(\SU(2)^2\).  Proposition~\ref{prop:powerproductcriterion}
gives a sufficient condition, namely that the ordered words
\(x^a y^b\) and \(x^c y^d\) form a free basis of \(F_2\), but the argument does not show whether this condition is necessary after
evaluation on $\SU(2)$.
The extension of \(M_A\) to quaternionic variables does not immediately
reduce the problem to the Dieudonn\'e determinant, since its differential
is generally real-linear rather than quaternionic-linear.  The general invertibility classification is left open.
\end{remark}

\section{Invariant maximal tori}

Recall that, in a compact Lie group, a \emph{maximal torus} is a maximal
connected abelian subgroup; equivalently, it is a connected compact
abelian subgroup that is not properly contained in any larger connected
abelian subgroup.  Every maximal torus of \(\SU(2)\) is one-dimensional.
More concretely, if \(I\) is a unit purely imaginary quaternion, then
\[
 S_I^1=\{e^{I\theta}:\theta\in\R\}\subset\SU(2)
\]
is a maximal torus, and every maximal torus of \(\SU(2)\) is of this
form.  Set
\[
 \T_I^2=S_I^1\times S_I^1\subset\SU(2)^2.
\]
Thus \(\T_I^2\) consists of pairs whose two components lie in the same
maximal torus \(S_I^1\) of \(\SU(2)\).  We say that such a pair lies in a \emph{common maximal torus}.
Thus \(\T_I^2\) is the product of that common maximal torus with itself.

Two elements \(q_1,q_2\in\SU(2)\) are said to form a \emph{commuting
pair} if
\[
 q_1q_2=q_2q_1.
\]
Equivalently, the subgroup generated by \(q_1\) and \(q_2\) is abelian.
For \(\SU(2)\), this is also equivalent to the existence of a maximal
torus \(S_I^1\subset\SU(2)\) containing both elements.  Such a maximal
torus is unique unless both elements are central.

\begin{theorem}\label{thm:torusrestriction}
Let \(F_\Phi\) be a quaternionic lift of
\(A=\left(\begin{smallmatrix}a&b\\c&d\end{smallmatrix}\right)\).
Then every \(\T_I^2\) is invariant and, in angular coordinates,
\[
 F_\Phi|_{\T_I^2}(\theta_1,\theta_2)
 =A\begin{pmatrix}\theta_1\\\theta_2\end{pmatrix}
 \pmod{2\pi\Z^2}.
\]
\end{theorem}

\begin{proof}
Elements of \(S_I^1\) commute.  Therefore word evaluation on
\(S_I^1\) depends only on total exponent sums.  Since these exponent sums are precisely the rows of
\(\rhoexp(\Phi)=A\), the two word coordinates reduce to
\[
 e^{I(a\theta_1+b\theta_2)},\qquad
 e^{I(c\theta_1+d\theta_2)}.
\]
\end{proof}

The union of the subgroups \(\T_I^2=S_I^1\times S_I^1\), obtained by
using the same maximal torus in both factors, is the commuting locus
\begin{equation}\label{eq:commutinglocus}
 \mathcal C=\{(q_1,q_2)\in\SU(2)^2:q_1q_2=q_2q_1\}.
\end{equation}
It is a closed, invariant, four-dimensional subset of the six-dimensional
ambient manifold.  Its only singular points are the four central pairs
\((\varepsilon_1,\varepsilon_2)\), with \(\varepsilon_i\in\{\pm1\}\).

\subsection{The pillowcase fibration of the commuting locus}
\label{subsec:pillowcase-fibration}
Every commuting pair can be written as
\[
 (q_1,q_2)=\bigl(e^{u\theta},e^{u\varphi}\bigr),
 \qquad u\in\Sph^2\subset\operatorname{Im}\mathbb H,
 \quad (\theta,\varphi)\in\T^2.
\]
This representation is subject to the Weyl identification
\begin{equation}\label{eq:weyl-identification}
 (u,\theta,\varphi)\sim(-u,-\theta,-\varphi).
\end{equation}
Let
\[
 \mathcal P:=\T^2/\{\xi\sim-\xi\}.
\]
The quotient \(\mathcal P\) is called the \emph{pillowcase}.  Its
underlying topological space is a two-sphere, while as an orbifold it has
four order-two points, namely the images of
\[
 \T^2[2]=\{(0,0),(\pi,0),(0,\pi),(\pi,\pi)\}.
\]
These four points will be denoted by \(p_1,p_2,p_3,p_4\).

Define the projection
\begin{equation}\label{eq:pillowcase-projection}
 \pi_{\mathcal C}:\mathcal C\longrightarrow\mathcal P,
 \qquad
 \pi_{\mathcal C}\bigl(e^{u\theta},e^{u\varphi}\bigr)
   =[\theta,\varphi].
\end{equation}
It is well defined precisely because of
\eqref{eq:weyl-identification}.  If \([\theta,\varphi]\) is not one of
the four orbifold points, then the fibre is the simultaneous-conjugacy
orbit
\[
 \pi_{\mathcal C}^{-1}([\theta,\varphi])
 =\left\{\bigl(e^{v\theta},e^{v\varphi}\bigr):v\in\Sph^2\right\}
 \cong\Sph^2.
\]
Over the four orbifold points the apparent sphere fibre collapses to one
of the four central pairs \((\pm1,\pm1)\).

\begin{proposition}[The punctured-pillowcase fibration]\label{prop:pillowcase-fibration}
Let
\[
 \T^2_*:=\T^2\setminus\T^2[2],
 \qquad
 \mathcal P^{\circ}:=\mathcal P\setminus\{p_1,p_2,p_3,p_4\}.
\]
Since the involution \(\xi\mapsto-\xi\) has precisely the four fixed
points \(\T^2[2]\), it acts freely on \(\T^2_*\).  Consequently,
\begin{equation}\label{eq:punctured-pillowcase}
 \mathcal P^{\circ}=\T^2_*/\{\xi\sim-\xi\}
 \cong S^2\setminus\{p_1,p_2,p_3,p_4\},
\end{equation}
and the quotient map
\[
 \rho:\T^2_*\longrightarrow\mathcal P^{\circ}
\]
is an honest two-sheeted covering.  If
\[
 \mathcal C_{\mathrm{reg}}
 :=\mathcal C\setminus\{(\pm1,\pm1)\},
\]
then
\begin{equation}\label{eq:regular-commuting-quotient}
 \mathcal C_{\mathrm{reg}}
 \cong
 \frac{S^2\times\T^2_*}{(u,\xi)\sim(-u,-\xi)},
\end{equation}
and the map
\begin{equation}\label{eq:regular-pillow-bundle}
 \pi_{\mathcal C}:\mathcal C_{\mathrm{reg}}
 \longrightarrow\mathcal P^{\circ},
 \qquad [u,\xi]\longmapsto[\xi],
\end{equation}
is an \(S^2\)-bundle.  Its pullback by \(\rho\) is the trivial bundle
\[
 \rho^*\mathcal C_{\mathrm{reg}}\cong S^2\times\T^2_*.
\]
The full commuting locus is obtained by restoring the four omitted
pillowcase points and collapsing the sphere over each of them to the
corresponding central pair.
\end{proposition}

\begin{proof}
Define
\[
 \Psi:S^2\times\T^2\longrightarrow\mathcal C,
 \qquad
 \Psi(u,\theta,\varphi)
   =\bigl(e^{u\theta},e^{u\varphi}\bigr).
\]
The two components of every commuting pair lie in a common maximal
torus \(S_I^1\subset\SU(2)\), so \(\Psi\) is
surjective.  If \((\theta,\varphi)\notin\T^2[2]\), at least one component
is noncentral and therefore determines the common imaginary axis up to
sign.  It follows that, on \(S^2\times\T^2_*\), the only ambiguity is
\[
 (u,\xi)\sim(-u,-\xi).
\]
This proves \eqref{eq:regular-commuting-quotient}.

Now the involution \(\xi\mapsto-\xi\) is free on \(\T^2_*\), so its
quotient is a genuine surface.  The quotient of the whole torus is the
pillowcase, whose underlying topological space is \(S^2\); deleting its
four branch points is exactly the same as deleting \(\T^2[2]\) before
taking the quotient.  This proves \eqref{eq:punctured-pillowcase}.

For a point \([\xi]\in\mathcal P^{\circ}\), choose either lift
\(\xi\in\T^2_*\).  Every class over \([\xi]\) has a unique representative
of the form \([u,\xi]\), and changing the chosen lift from \(\xi\) to
\(-\xi\) changes the fibre coordinate from \(u\) to \(-u\).  Thus the
fibre is a copy of \(S^2\), and local sections of the two-sheeted covering
\(\rho\) give local trivializations of \(\pi_{\mathcal C}\).  Pulling the
bundle back to \(\T^2_*\) removes the sign ambiguity and gives the product
\(S^2\times\T^2_*\).

Finally, over a point of \(\T^2[2]\), both exponentials are central and
independent of \(u\).  Hence the entire apparent \(S^2\)-fibre maps to one
central pair.  This is precisely the stated collapse of the four
exceptional fibres.
\end{proof}

Thus
\[
 S^2\longrightarrow \mathcal C_{\mathrm{reg}}
 \xrightarrow{\ \pi_{\mathcal C}\ }
 \mathcal P^\circ
 \cong S^2\setminus\{p_1,p_2,p_3,p_4\}.
\]
The bundle has nontrivial monodromy around each puncture.

\begin{proposition}\label{prop:commutinglocaltopology}
The monodromy of the conjugacy-sphere bundle around each puncture of the
regular pillowcase is the antipodal map
\[
 u\longmapsto -u
\]
of \(S^2\).  Hence the monodromy reverses the orientation of the fibre.
At each central pair, the link of the singularity is the mapping torus
\[
 L=\frac{S^2\times[0,1]}{(u,1)\sim(-u,0)}
 \cong
 \frac{S^2\times S^1}{(u,z)\sim(-u,-z)}.
\]
In particular,
\[
 \pi_1(L)\cong\mathbb Z,
 \qquad H_2(L;\mathbb Z)\cong\mathbb Z/2,
\]
so \(L\) is not homeomorphic to \(S^3\).  Therefore \(\mathcal C\) is
not a topological manifold at the four central pairs.
\end{proposition}

\begin{proof}
It suffices to work near \((1,1)\), since multiplication by central elements gives the same
model at the other three central pairs.  In exponential coordinates write
\[
 q_1=\exp X,\qquad q_2=\exp Y,
 \qquad X,Y\in\mathfrak{su}(2)\cong\mathbb R^3.
\]
For \(X,Y\) in the open ball of radius \(\pi\), on which \(\exp\) is
injective, this is exact rather than merely infinitesimal: if \(X\neq0\)
then \(\exp X\neq\pm1\), so its centralizer in \(\SU(2)\) is exactly the
maximal torus through \(\exp X\), and \(\exp Y\) lies in this centralizer
if and only if \(Y\) lies on the same axis as \(X\).  Hence, for such
\(X,Y\), the two group elements commute if and
only if \([X,Y]=0\), equivalently \(X\times Y=0\).  Thus the local model is
\[
 V=\{(X,Y)\in\mathbb R^3\times\mathbb R^3:X\times Y=0\}.
\]
Every nonzero point of \(V\) can be written as
\[
 (X,Y)=(au,bu),
 \qquad u\in S^2,
 \quad (a,b)\in\mathbb R^2\setminus\{0\},
\]
with the unique ambiguity
\[
 (u,a,b)\sim(-u,-a,-b).
\]
The punctured local base is therefore
\((\mathbb R^2\setminus\{0\})/\{v\sim-v\}\).  A loop going once around
its puncture lifts to a path in \(\mathbb R^2\setminus\{0\}\) beginning
at \(v_0\) and ending at \(-v_0\).  Transporting \(u\) constantly along
this lifted path gives \((u,-v_0)\) at the endpoint, which is identified
with \((-u,v_0)\).  Hence the fibre return map is exactly the antipodal
map \(u\mapsto-u\).  Since the antipodal map on \(S^2\) has degree \(-1\),
it reverses orientation.

Intersecting \(V\) with a small Euclidean \(5\)-sphere imposes
\(a^2+b^2=1\), and therefore gives
\[
 L\cong(S^2\times S^1)/\bigl((u,z)\sim(-u,-z)\bigr),
\]
which is the mapping torus of the antipodal map.  Its fundamental group
is \(\mathbb Z\).  The Wang exact sequence for the mapping torus, using
that the antipodal map acts as \(-1\) on \(H_2(S^2;\mathbb Z)\), gives
\(H_2(L;\mathbb Z)\cong\mathbb Z/2\).  Since a neighborhood of the central
point is the cone on \(L\), its local homology satisfies
\[
 H_3(\mathcal C,\mathcal C\setminus\{(1,1)\};\mathbb Z)
 \cong \widetilde H_2(L;\mathbb Z)\cong\mathbb Z/2.
\]
For a topological four-manifold the local homology is zero in degree \(3\).
Therefore the central point is not a manifold point.
\end{proof}

Thus \(\mathcal C\) may be viewed as a twisted conjugacy-sphere bundle
over the four-punctured sphere, completed by collapsing the fibre over
each missing point to one of the four central pairs.  The antipodal
monodromy is the precise reason this completion is singular rather than
an ordinary \(S^2\)-bundle over \(S^2\).

\section{The quaternionic Fibonacci cat map}

Consider the hyperbolic matrix
\[
 A=\begin{pmatrix}2&1\\1&1\end{pmatrix},
 \qquad
 \lambda_\pm=\frac{3\pm\sqrt5}{2}.
\]
It is the product
\[
 A=
 \begin{pmatrix}1&1\\0&1\end{pmatrix}
 \begin{pmatrix}1&0\\1&1\end{pmatrix}.
\]
The corresponding composition of Nielsen maps is
\begin{equation}\label{eq:catmap}
 F(q_1,q_2)=\bigl(q_1^2q_2,q_1q_2\bigr).
\end{equation}
This map also has power-product form.

\begin{proposition}\label{prop:catinverse}
The map \eqref{eq:catmap} is a real-analytic diffeomorphism with
\[
 F^{-1}(u,v)=\bigl(uv^{-1},vu^{-1}v\bigr).
\]
In general this is not the power-product expression associated with
\(A^{-1}\).
\end{proposition}

\begin{proof}
If \(u=q_1^2q_2\) and \(v=q_1q_2\), then
\[
 uv^{-1}=q_1^2q_2(q_1q_2)^{-1}=q_1.
\]
It follows that
\[
 q_2=q_1^{-1}v=(uv^{-1})^{-1}v=vu^{-1}v.
\]
Direct substitution proves the assertion.
\end{proof}

The inverse matrix is
\[
 A^{-1}=\begin{pmatrix}1&-1\\-1&2\end{pmatrix},
\]
whose naive power-product expression would be
\((uv^{-1},u^{-1}v^2)\), generally different from the actual inverse.
The two inverse formulas differ because quaternionic multiplication is noncommutative.

\section{Invariant volume}

Let \(m\) be normalized Haar measure on \(\SU(2)\), and let
\(m\times m\) be product Haar measure.  Let \(\nu\) be the bi-invariant
Riemannian volume form on \(\SU(2)\) and
\[
 \Omega=\pi_1^*\nu\wedge\pi_2^*\nu.
\]

\begin{theorem}\label{thm:volume}
Every quaternionic Nielsen map \(F_\Phi\) preserves product Haar measure.
If \(A=\rhoexp(\Phi)\), then
\[
 F_\Phi^*\Omega=(\det A)\Omega.
\]
In particular, every lift of a matrix in \(\mathrm{SL}(2,\Z)\) is
orientation preserving and volume preserving.
\end{theorem}

\begin{proof}
It suffices to check elementary Nielsen generators.  For example, for
\(N_{12}(q_1,q_2)=(q_1q_2,q_2)\) and every continuous test function
\(\varphi\), right invariance of Haar measure gives
\begin{align*}
 \int\varphi(q_1q_2,q_2)\,dm(q_1)dm(q_2)
 &=\int\varphi(u,q_2)\,dm(u)dm(q_2).
\end{align*}
The proof for \(N_{21}\) is identical.  Inversion and permutation also
preserve Haar measure.  Therefore every composition preserves product
Haar measure.

For orientations, the two shear Nielsen maps are orientation preserving.
Inversion on the three-dimensional group \(\SU(2)\) reverses orientation,
and interchanging the two three-dimensional factors also reverses
orientation.  These signs agree with the determinants of the corresponding row
exponent matrices (equivalently, of the standard abelianization
matrices).  Multiplicativity gives the stated formula.
\end{proof}

\begin{corollary}[Poincar\'e recurrence]\label{cor:recurrence}
For every quaternionic Nielsen map $F_\Phi$, almost every point of
$\SU(2)^2$, with respect to product Haar measure $m\times m$, is
recurrent.  More precisely, for $(m\times m)$-almost every
$p\in\SU(2)^2$ there exists a sequence $n_j\to\infty$ such that
\[
 F_\Phi^{n_j}(p)\longrightarrow p.
\]
Equivalently, if $E\subset\SU(2)^2$ is measurable, then almost every
point of $E$ returns to $E$ infinitely many times.
\end{corollary}

\begin{proof}
The space $\SU(2)^2$ is compact and $m\times m$ is a finite invariant
measure.  The conclusion is therefore the Poincar\'e recurrence theorem
applied to the measure-preserving transformation $F_\Phi$; see, for
example, \cite{KatokHasselblatt}.
\end{proof}

For the cat map \eqref{eq:catmap},
\[
 \Jac_\Omega F\equiv1.
\]
\section{Entropy and periodic points}

Suppose that \(A\in\mathrm{SL}(2,\Z)\) is hyperbolic, and let
\(\rho(A)>1\) denote its spectral radius.  The invariant maximal tori give
the elementary lower bound \(h_{\mathrm{top}}(F_\Phi)\geq\log\rho(A)\).
Fried proved that this lower bound is exact for every invertible word map on
\(S^3\times S^3\) \cite[Theorem~1]{Fried}.

\begin{theorem}[Fried]\label{thm:entropy}
For every quaternionic Nielsen lift \(F_\Phi\) with row exponent matrix
\(A\in\mathrm{GL}(2,\Z)\),
\[
 h_{\mathrm{top}}(F_\Phi)=\log\rho(A).
\]
In particular, for the Fibonacci cat map,
\[
 h_{\mathrm{top}}(F)=
 \log\left(\frac{3+\sqrt5}{2}\right).
\]
\end{theorem}

Thus the noncommuting directions do not create additional topological
entropy in rank two.

\begin{question}\label{q:metric-entropy}
What is the metric entropy of \(F_\Phi\) with respect to product Haar measure
when the abelianization matrix is hyperbolic?  Fried's theorem determines the
topological entropy.  Gogolev and Rodr\'iguez Hertz raised the metric-entropy
problem in their unpublished notes.
\end{question}

Periodic points of a hyperbolic toral automorphism are dense.  More
precisely,
\[
 \#\Fix(A^n)=|\det(A^n-I)|.
\]
Applying this on every common maximal torus proves density on the
regular part of the commuting locus.  The four central pairs are also in
the closure: every neighborhood of a central pair in \(\mathcal C\)
contains points on nearby common maximal tori, and on each such torus
periodic points are dense.  Hence:

\begin{corollary}\label{cor:densecommuting}
For a quaternionic lift of a hyperbolic matrix, periodic points are dense
in the whole commuting locus \(\mathcal C\), including its four singular
central points.  For the Fibonacci cat map,
\[
 \#\Fix(A^n)=\lambda_+^n+\lambda_+^{-n}-2,
\]
so the toral periodic data grow exponentially with rate
\(\log\lambda_+\).
\end{corollary}

The commuting locus is a proper closed subset of \(\SU(2)^2\), and this
argument does not prove density of periodic points in the whole
six-manifold.

\section{Conjugation symmetry and periodic two-spheres}

The group \(\SU(2)\) acts on \(\SU(2)^2\) by simultaneous conjugation:
\[
 h\cdot(q_1,q_2)=(hq_1h^{-1},hq_2h^{-1}).
\]

\begin{proposition}\label{prop:equivariance}
Every word map \(F_\Phi\) is equivariant for simultaneous conjugation:
\[
 F_\Phi(h\cdot(q_1,q_2))=h\cdot F_\Phi(q_1,q_2).
\]
\end{proposition}

\begin{proof}
For every group word \(w\),
\[
 w(hq_1h^{-1},hq_2h^{-1})=h w(q_1,q_2)h^{-1}.
\]
Apply this identity to the two coordinate words of \(F_\Phi\).
\end{proof}

\subsection*{Periodic conjugacy spheres}
The expression \emph{periodic conjugacy sphere} means something stronger
than a sphere which merely returns to itself periodically.  It is a
simultaneous-conjugation orbit on which the appropriate return map is
pointwise the identity.  In particular, it is not a sphere on which the
map acts by a nontrivial periodic rotation.  Let \(p=(q_1,q_2)\) be a noncentral commuting
periodic point of least period \(n\), and let
\[
 \mathcal O_p=
 \{(hq_1h^{-1},hq_2h^{-1}):h\in\SU(2)\}
\]
be its simultaneous-conjugation orbit.  Equivariance gives
\[
 F_\Phi^n(h\cdot p)=h\cdot F_\Phi^n(p)=h\cdot p,
\]
and therefore
\[
 F_\Phi^n|_{\mathcal O_p}=\operatorname{id}_{\mathcal O_p}.
\]
Thus the return map fixes the entire conjugacy orbit pointwise, not
just the original point.

For a generic commuting noncentral pair, the common stabilizer is the
maximal torus \(S^1\subset\SU(2)\) containing both components.  The
orbit--stabilizer theorem for Lie-group actions therefore gives
\[
 \mathcal O_p\cong\SU(2)/S^1\cong\Sph^2.
\]
Such an orbit will be called a \emph{periodic conjugacy sphere}.  If \(p\) has exact
period \(n>1\), the map may cyclically permute
\(\mathcal O_p,\mathcal O_{F_\Phi(p)},\ldots,
\mathcal O_{F_\Phi^{n-1}(p)}\), while its \(n\)-th iterate fixes each sphere pointwise.  Thus
\(F_\Phi\) may permute a cycle of spheres, but the first-return map on
each member of the cycle is exactly the identity.  In the character-variety quotient the whole sphere
collapses to the single periodic class \([p]\).

\begin{theorem}\label{thm:notanosov}
A quaternionic lift of a hyperbolic matrix is not an Anosov diffeomorphism
of \(\SU(2)^2\).  It is also not \(C^1\)-structurally stable.
\end{theorem}

\begin{proof}
Choose a noncentral periodic point \(p\) on an invariant maximal torus.
If its period is \(n\), equivariance implies that \(F_\Phi^n\) fixes the
entire conjugation orbit \(\mathcal O_p\cong\Sph^2\) pointwise.  Hence
\[
 D(F_\Phi^n)_p v=v
 \qquad(v\in T_p\mathcal O_p).
\]
The derivative of the return map therefore has eigenvalue \(1\).  The
periodic point is not hyperbolic, whereas every periodic point of an
Anosov diffeomorphism is hyperbolic.  Thus \(F_\Phi\) is not Anosov.

By Ma\~n\'e's proof of the \(C^1\)-stability conjecture
\cite{Mane}, a \(C^1\)-structurally stable diffeomorphism satisfies
Axiom~A and the strong transversality condition.  Here Axiom~A means that the nonwandering
set is hyperbolic and that its periodic points are dense in it; in
particular all periodic points are hyperbolic.  The same periodic sphere therefore obstructs
structural stability.
\end{proof}

\section{Character varieties and reduced dynamics}

The group \(\Aut(F_2)\) acts by precomposition on representations.  A point of
\(\Hom(F_2,\SU(2))\) is a homomorphism \(\rho:F_2\to\SU(2)\).
For \(\Phi\in\Aut(F_2)\), define
\begin{equation}\label{eq:precompositionaction}
 \Phi^*(\rho):=\rho\circ\Phi.
\end{equation}
Under the identification
\[
 \Hom(F_2,\SU(2))\longrightarrow\SU(2)^2,
 \qquad \rho\longmapsto(\rho(x),\rho(y)),
\]
the map \(\Phi^*\) is precisely the word-evaluation map \(F_\Phi\):
\[
 \Phi^*(\rho)\longleftrightarrow
 \bigl(\rho(\Phi(x)),\rho(\Phi(y))\bigr)=F_\Phi(q_1,q_2).
\]
This is a right action with our convention, since
\((\Phi\circ\Psi)^*=\Psi^*\circ\Phi^*\).

The group \(\SU(2)\) acts on representations by conjugation,
\[
 (h\cdot\rho)(\gamma)=h\rho(\gamma)h^{-1}.
\]
Precomposition commutes with this action:
\[
 \Phi^*(h\cdot\rho)(\gamma)
 =h\rho(\Phi(\gamma))h^{-1}
 =\bigl(h\cdot\Phi^*(\rho)\bigr)(\gamma).
\]
Consequently \(\Phi^*\) sends simultaneous-conjugacy orbits to
simultaneous-conjugacy orbits and therefore descends to a well-defined
map
\begin{equation}\label{eq:characterdescent}
 \overline F_\Phi:\mathcal X_{\SU(2)}(F_2)
 \longrightarrow\mathcal X_{\SU(2)}(F_2),
 \qquad [\rho]\longmapsto[\rho\circ\Phi].
\end{equation}
Indeed, replacing \(\rho\) by the conjugate representation
\(h\cdot\rho\) replaces \(\rho\circ\Phi\) by the conjugate
representation \(h\cdot(\rho\circ\Phi)\), so the class on the right of
\eqref{eq:characterdescent} is independent of the representative.

If \(\iota_w\) is the inner
automorphism \(\gamma\mapsto w\gamma w^{-1}\), then
\[
 \rho\circ\iota_w
 =\rho(w)\,\rho(\,\cdot\,)\,\rho(w)^{-1},
\]
so \(\iota_w\) acts trivially on the character variety.  Hence the
\(\Aut(F_2)\)-action factors through
\[
 \Out(F_2)=\Aut(F_2)/\operatorname{Inn}(F_2)\cong\mathrm{GL}(2,\Z).
\]
Thus a matrix \(A\in\mathrm{GL}(2,\Z)\) determines a canonical reduced map on
the character variety even though its Nielsen lift to \(\SU(2)^2\) is
not unique.

The \(\SU(2)\)-character variety of \(F_2\) is the orbit space
\[
 \mathcal X_{\SU(2)}(F_2)
 :=\Hom(F_2,\SU(2))/\SU(2)
 \cong\SU(2)^2/\SU(2),
\]
where \(\SU(2)\) acts by simultaneous conjugation.  Its points are thus
conjugacy classes \([q_1,q_2]\) of representations of \(F_2\).
Normalized trace coordinates are
\[
 x=\tfrac12\tr(q_1),\qquad
 y=\tfrac12\tr(q_2),\qquad
 z=\tfrac12\tr(q_1q_2).
\]
They identify the character variety with the compact region
(see \cite{Fried,Goldman,GoldmanMcShaneStantchevTan})
\begin{equation}\label{eq:characterball}
 \mathcal B=
 \{(x,y,z)\in[-1,1]^3:
 x^2+y^2+z^2-2xyz\leq1\},
\end{equation}
which is topologically a closed three-ball.  Define the \emph{Fricke
function}
\begin{equation}\label{eq:fricke}
 \Delta(x,y,z)=1-x^2-y^2-z^2+2xyz.
\end{equation}
The boundary \(\partial\mathcal B=\{\Delta=0\}\) consists exactly of
characters of commuting pairs.  The terminology comes from the Fricke
trace identity
\[
 \tr[q_1,q_2]=2-4\Delta(x,y,z),
 \qquad [q_1,q_2]=q_1q_2q_1^{-1}q_2^{-1}.
\]

By Proposition~\ref{prop:nonunique}, the induced map on
\(\mathcal B\) depends only on the matrix \(A\), not on the chosen
Nielsen lift.  Fried described this polynomial \(\mathrm{GL}(2,\mathbb Z)\)-action and its invariant Fricke level sets in \cite{Fried}.  For the Fibonacci cat
map \eqref{eq:catmap}, in the present trace normalization, direct use of
\(\tr(UV)+\tr(UV^{-1})=\tr(U)\tr(V)\) gives the polynomial automorphism
\begin{equation}\label{eq:reducedcat}
 T(x,y,z)=\bigl(2xz-y,\ z,\ 4xz^2-2yz-x\bigr).
\end{equation}

\begin{proposition}\label{prop:frickeinvariant}
The reduced map \(T\) preserves the Fricke function:
\[
 \Delta\circ T=\Delta.
\]
Consequently every level set
\[
 \mathcal S_\delta=\{\Delta=\delta\},\qquad0\leq\delta\leq1,
\]
is invariant.
\end{proposition}

\begin{proof}
The identity follows either by substituting \eqref{eq:reducedcat} into
\eqref{eq:fricke}, or from the commutator trace.  An
orientation-preserving automorphism of \(F_2\) sends \([x,y]\) to a
conjugate of itself, so its trace is unchanged.
\end{proof}

The boundary \(\mathcal S_0\) is the \emph{pillowcase}: the quotient
\[
 \T^2/\{\xi\sim-\xi\}.
\]
It is topologically a two-sphere with four orbifold points, the images of
the four fixed points of the involution \(\xi\mapsto-\xi\).  The cat
matrix commutes with this involution and therefore induces a map on the
pillowcase.

\begin{theorem}\label{thm:pillowcase}
The restriction of \(T\) to \(\partial\mathcal B\) is the pillowcase
quotient of the classical cat map
\[
 A=\begin{pmatrix}2&1\\1&1\end{pmatrix}:\T^2\longrightarrow\T^2.
\]
It has topological entropy
\[
 \log\left(\frac{3+\sqrt5}{2}\right)
\]
and dense periodic points.  In the orbifold sense it is pseudo-Anosov:
its stable and unstable measured foliations are the quotients of the
linear eigendirections of \(A\).
\end{theorem}

\begin{proof}
A commuting pair is conjugate to
\((e^{I\theta},e^{I\varphi})\).  Replacing \(I\) by \(-I\) identifies
\((\theta,\varphi)\) with \((-\theta,-\varphi)\), producing the
pillowcase quotient.  The restriction of the lift is the cat map by
Theorem~\ref{thm:torusrestriction}.  Entropy and density of periodic
points pass through the finite quotient.
\end{proof}

The full character-ball dynamics is not uniformly hyperbolic.  For example, the
fixed-point equations for \eqref{eq:reducedcat} contain the curve
\[
 p_t=\left(\frac{t}{2t-1},t,t\right)
\]
whenever this point belongs to \(\mathcal B\).  At \(p_0=(0,0,0)\), the
derivative has eigenvalues
\[
 1,\qquad e^{2\pi i/3},\qquad e^{-2\pi i/3}.
\]
Thus the derivative at this fixed point has no hyperbolic splitting.  Since \(p_0\) lies on the degenerate level \(\mathcal S_1=\{(0,0,0)\}\), this statement does not concern a two-dimensional level surface.

Fried already raised the ergodicity problem for the intermediate
Fricke level spheres \cite{Fried}.  Goldman and Goldman--Xia proved
ergodicity results for mapping-class-group actions on suitable relative
\(\SU(2)\)-character varieties \cite{GoldmanErgodic,GoldmanXia2011}.
These results concern the group action and do not imply ergodicity of an
individual hyperbolic element.

\begin{question}
What are the ergodic properties of \(T|_{\mathcal S_\delta}\) for
\(0<\delta<1\)?  Do elliptic periodic points occur on these levels, and
what nonuniformly hyperbolic invariant sets occur as \(\delta\) varies?
\end{question}

\begin{question}
Can one select a dynamically preferred representative of the canonical
outer automorphism class \([\Phi_A]\), despite the nonuniqueness of the
lift on \(\SU(2)^2\)?
\end{question}

\section{Homological growth and stable/unstable eigencurrents}
\label{sec:eigencurrent}

Let
\[
 C_1=S^3\times\{1\},\qquad C_2=\{1\}\times S^3,
\]
with their standard orientations.  Their fundamental classes form the
natural basis of
\[
 H_3(S^3\times S^3;\mathbb Z)\cong\mathbb Z^2.
\]
If \(F_A\) is a Nielsen lift of \(A\in\mathrm{SL}(2,\mathbb Z)\), then
\[
 (F_A)_*=A
 \quad\text{on }H_3(S^3\times S^3;\mathbb Z)
\]
with respect to this basis.  Indeed, if the two defining words have exponent-sum rows
\((a,b)\) and \((c,d)\), then
\[
 F_A(q,1)=(q^a,q^c),\qquad F_A(1,q)=(q^b,q^d),
\]
and the power map \(q\mapsto q^m\) has degree \(m\) on \(S^3\).  Indeed,
if \(\eta\) generates \(H^3(S^3;\mathbb Z)\), the multiplication map of the
Lie group satisfies
\(\mathrm{mult}^*\eta=\eta\otimes1+1\otimes\eta\); iteration gives
\((q\mapsto q^m)^*\eta=m\eta\), with the same formula for negative \(m\)
using inversion.

\subsection{Embedded three-spheres representing homology classes}

The same power maps give explicit representatives of primitive classes in
middle-dimensional homology.  For integers $n,m$, define
\[
 e_{n,m}:S^3\longrightarrow S^3\times S^3,
 \qquad
 e_{n,m}(q)=(q^n,q^m).
\]
Since the two coordinate projections have degrees $n$ and $m$,
\begin{equation}\label{eq:enm-homology}
 (e_{n,m})_*[S^3]=n[C_1]+m[C_2].
\end{equation}
The map $e_{n,m}$ is not always an embedding.  

\begin{proposition}\label{prop:embedded-H3-representatives}
Every class
\[
 \alpha=n[C_1]+m[C_2]\in H_3(S^3\times S^3;\mathbb Z)
\]
is represented by a smooth oriented embedded three-sphere.  If
$\gcd(|n|,|m|)=1$, the map $e_{n,m}$ itself is a smooth embedding.  For a
nonprimitive nonzero class, writing
\[
 (n,m)=d(a,b),\qquad d=\gcd(|n|,|m|),\qquad \gcd(|a|,|b|)=1,
\]
one obtains an embedded representative by taking $d$ parallel copies of
$e_{a,b}(S^3)$ and joining them by oriented tubes.  The zero class is
represented by a small embedded $S^3$ which bounds a $4$-ball in a coordinate
ball of $S^3\times S^3$.
\end{proposition}

\begin{proof}
Equation~\eqref{eq:enm-homology} follows from the degree of the power maps.
Assume first that $\gcd(|n|,|m|)=1$.  Choose $r,s\in\mathbb Z$ with
\[
 rn+sm=1.
\]
The smooth map
\[
 L:S^3\times S^3\longrightarrow S^3,
 \qquad L(x,y)=x^r y^s,
\]
satisfies
\[
 L\circ e_{n,m}(q)=q^{rn}q^{sm}=q.
\]
Thus $e_{n,m}$ is injective and its differential is injective.  Since $S^3$
is compact, it is a smooth embedding.

Now let $(n,m)=d(a,b)$ with $d>1$ and $(a,b)$ primitive.  The normal bundle
of the embedded sphere $e_{a,b}(S^3)$ is an oriented rank-three bundle over
$S^3$, hence is trivial because $\pi_2(SO(3))=0$.  A tubular neighborhood
therefore contains $d$ disjoint parallel copies of this sphere.  Give them
the same orientation and join them successively by disjoint oriented tubes.
The resulting connected submanifold is the connected sum of $d$ copies of
$S^3$, hence is diffeomorphic to $S^3$, and its homology class is
$d(a[C_1]+b[C_2])=\alpha$.  The case $\alpha=0$ is obtained from the
standard $S^3=\partial D^4$ inside a coordinate $D^6$.
\end{proof}

\begin{remark}
For a nonprimitive pair $(n,m)$, the formula $q\mapsto(q^n,q^m)$ still
represents the class $(n,m)$ as a smooth map, but it is generally not an
embedding.  If $d=\gcd(|n|,|m|)>1$, then
\[
 e_{n,m}=e_{n/d,m/d}\circ p_d,
 \qquad p_d(q)=q^d,
\]
and $p_d$ is not injective.
\end{remark}

For hyperbolic \(A\), normalized forward iterates of the coordinate cycles converge as de Rham currents to an unstable eigen-current supported on the commuting locus.  The inverse map gives the stable eigencurrent.  Their one-dimensional spans are invariant under push-forward, with eigenvalues \(\mu_u\) and \(\mu_u^{-1}\), respectively.

Throughout this section, currents are understood in the weak de Rham
sense: a sequence of \(k\)-currents converges if it converges on every
smooth \(k\)-form.

\begin{theorem}[Unstable eigencurrent]
\label{thm:unstable-current}
Let \(A\in\mathrm{SL}(2,\mathbb Z)\) be hyperbolic.  Let \(\mu_u\) be
its unstable eigenvalue, so that \(|\mu_u|>1\), and choose an unstable
eigenvector \(v^u\in\mathbb R^2\) with \(Av^u=\mu_u v^u\).
Let \(F_A\) be any Nielsen lift of \(A\).  For \(i=1,2\), write
\[
 \mu_u^{-n}A^ne_i\longrightarrow c_i v^u.
\]
At least one of the coefficients \(c_i\) is nonzero.  Then there exists a
nonzero closed normal \(3\)-current of finite mass
\(\mathcal T^u\), supported on the commuting locus \(\mathcal C\), such that
\[
 \mu_u^{-n}(F_A^n)_\#[C_i]
 \longrightarrow c_i\mathcal T^u
\]
weakly as de Rham currents.  Moreover,
\[
 (F_A)_\#\mathcal T^u=\mu_u\mathcal T^u,
\]
and its homology class is \(v^u\) in
\(H_3(S^3\times S^3;\mathbb R)\cong\mathbb R^2\).
\end{theorem}

The proof is based on an exact reduction to long rational geodesics on
the two-torus.

\subsection{Exact parametrization of the iterated cycles}

Let \(\Phi\in\operatorname{Aut}(F_2)\) be the automorphism defining the
chosen lift.  The row exponent matrix of \(\Phi^n\) is \(A^n\).  If
\[
 A^ne_i=\binom{m_{i,n}}{\ell_{i,n}},
\]
then setting one free generator equal to the identity makes every word
commutative and leaves only its exponent sum.  Consequently,
\[
 F_A^n(q,1)=\bigl(q^{m_{1,n}},q^{\ell_{1,n}}\bigr),
 \qquad
 F_A^n(1,q)=\bigl(q^{m_{2,n}},q^{\ell_{2,n}}\bigr).
\]
Thus, for \(v=(m,\ell)\in\mathbb Z^2\), if
\[
 \phi_v:S^3\longrightarrow S^3\times S^3,
 \qquad
 \phi_v(q)=(q^m,q^\ell),
\]
then
\[
 (F_A^n)_\#[C_i]=(\phi_{A^ne_i})_\#[S^3].
\]
In particular, every one of these currents is supported on
\(\mathcal C\).

Introduce the smooth maps
\[
 E:S^2\times S^1\longrightarrow S^3,
 \qquad E(u,t)=e^{tu}=\cos t+u\sin t,
\]
and
\[
 \Psi:S^2\times\mathbb T^2\longrightarrow\mathcal C,
 \qquad
 \Psi(u,\theta,\varphi)
   =\bigl(e^{u\theta},e^{u\varphi}\bigr),
\]
where \(S^1=\mathbb R/2\pi\mathbb Z\) and
\(\mathbb T^2=(\mathbb R/2\pi\mathbb Z)^2\).  The map \(E\) has degree
\(2\): away from \(\pm1\), its two inverse images are
\((u,t)\) and \((-u,-t)\), and the involution reverses the orientation
of each factor and hence preserves the product orientation.  Therefore
\[
 E_\#[S^2\times S^1]=2[S^3].
\]
For \(v=(m,\ell)\), define the closed geodesic
\[
 \gamma_v:S^1\longrightarrow\mathbb T^2,
 \qquad \gamma_v(t)=(mt,\ell t).
\]
Since
\[
 \phi_v\circ E
 =\Psi\circ(\operatorname{id}_{S^2}\times\gamma_v),
\]
we obtain the current identity
\begin{equation}
\label{eq:cycle-product-current}
 (\phi_v)_\#[S^3]
 =\frac12\,\Psi_\#\bigl([S^2]\times[\gamma_v]\bigr).
\end{equation}

\subsection{A Fourier lemma for rational geodesics}

For \(w=(w_1,w_2)\in\mathbb R^2\), define the constant
Kronecker current \(\mathcal R_w\) on \(\mathbb T^2\) by
\[
 \mathcal R_w(\alpha)
 =\frac1{2\pi}\int_{\mathbb T^2}
   \bigl(w_1\alpha_1+w_2\alpha_2\bigr)\,d\theta\,d\varphi,
 \qquad
 \alpha=\alpha_1\,d\theta+\alpha_2\,d\varphi.
\]
It is a closed normal $1$-current of finite mass and depends linearly on $w$.

\begin{lemma}
\label{lem:rational-geodesics}
Let \(v_n\in\mathbb Z^2\) be primitive vectors and let
\(s_n\in\mathbb R\setminus\{0\}\) satisfy
\[
 s_n^{-1}v_n\longrightarrow w,
 \qquad |s_n|\longrightarrow\infty,
\]
where the direction of \(w\) is irrational.  Then
\[
 s_n^{-1}[\gamma_{v_n}]\longrightarrow\mathcal R_w
\]
weakly as \(1\)-currents on \(\mathbb T^2\).
\end{lemma}

\begin{proof}
For a smooth one-form
\(\alpha=\alpha_1d\theta+\alpha_2d\varphi\), expand its coefficients in
Fourier series,
\[
 \alpha_j(x)=\sum_{k\in\mathbb Z^2}\widehat\alpha_j(k)e^{i\langle k,x\rangle}.
\]
Evaluation on \(\gamma_v\) gives
\[
 [\gamma_v](\alpha)
 =\int_0^{2\pi}\langle v,\alpha(tv)\rangle\,dt.
\]
The Fourier mode indexed by \(k\) integrates to zero unless
\(\langle k,v\rangle=0\).  The zero mode contributes
\[
 2\pi\langle v,\widehat\alpha(0)\rangle
 =\frac1{2\pi}\int_{\mathbb T^2}
   \langle v,\alpha\rangle\,d\theta\,d\varphi.
\]
Because \(v\) is primitive, every nonzero resonant vector is an integer
multiple of \((-v_2,v_1)\), and therefore has norm at least \(\|v\|\).
The rapid decay of the Fourier coefficients of a smooth form shows that,
for every \(N\), the sum of all nonzero resonant contributions is
\(O_N(\|v\|^{1-N})\).  Since \(v_n/s_n\to w\), the ratios \(\|v_n\|/|s_n|\) are bounded;
after division by \(s_n\), this error tends to zero, while the zero mode converges to
\(\mathcal R_w(\alpha)\).  This proves the assertion.
\end{proof}

\subsection{Proof of Theorem~\ref{thm:unstable-current}}

The vectors \(A^ne_i\) are primitive because \(A^n\) is unimodular.  Since
the unstable eigenvalue is simple and strictly dominates the stable one in
modulus, the spectral projection onto the unstable line gives
\[
 \mu_u^{-n}A^ne_i\longrightarrow c_i v^u.
\]
At least one \(c_i\) is nonzero, since the unstable spectral projection is
not the zero map.  The unstable slope of an integral hyperbolic matrix is
irrational.  If \(c_i\neq0\), Lemma~\ref{lem:rational-geodesics}, applied
with limiting vector \(c_i v^u\), gives
\[
 \mu_u^{-n}[\gamma_{A^ne_i}]
 \longrightarrow c_i\mathcal R_{v^u}.
\]
If \(c_i=0\), then \(A^ne_i\) lies asymptotically in the stable spectral
subspace, and hence
\(\|A^ne_i\|=o(|\mu_u|^n)\).  By
\eqref{eq:cycle-product-current}, the mass of
\(\mu_u^{-n}(\phi_{A^ne_i})_\#[S^3]\) is bounded by a fixed constant
times \(\|A^ne_i\|/|\mu_u|^n\), and therefore tends to zero.  Thus the
same current convergence holds also when \(c_i=0\).
Define
\begin{equation}
\label{eq:unstable-current-definition}
 \mathcal T^u
 :=\frac12\,\Psi_\#
 \bigl([S^2]\times\mathcal R_{v^u}\bigr).
\end{equation}
Combining this definition with
\eqref{eq:cycle-product-current} proves
\[
 \mu_u^{-n}(F_A^n)_\#[C_i]
 \longrightarrow c_i\mathcal T^u.
\]
The current $\mathcal R_{v^u}$ has finite mass and zero boundary on the compact torus.  Hence
$[S^2]\times\mathcal R_{v^u}$ is a normal current of finite mass.  Since
$\Psi$ is smooth on the compact manifold $S^2\times\mathbb T^2$, its
push-forward $\mathcal T^u$ is again normal and of finite mass; in
particular $\partial\mathcal T^u=0$.  See \cite{Federer}.

On the commuting locus, every Nielsen lift acts in angular coordinates
by \(A\):
\[
 F_A\circ\Psi
 =\Psi\circ(\operatorname{id}_{S^2}\times A).
\]
Consequently,
\[
 (F_A)_\#\mathcal T^u
 =\frac12\Psi_\#
   \bigl([S^2]\times A_\#\mathcal R_{v^u}\bigr)
 =\frac12\Psi_\#
   \bigl([S^2]\times\mathcal R_{Av^u}\bigr)
 =\mu_u\mathcal T^u.
\]
Finally,
\[
 [(F_A^n)_\#[C_i]]=A^ne_i.
\]
Choose an index \(i\) with \(c_i\neq0\).  Pairing the weak current limit
with closed \(3\)-forms gives
\(c_i[\mathcal T^u]=c_i v^u\), and hence
\[
 [\mathcal T^u]=v^u.
\]
In particular \(\mathcal T^u\neq0\).
\qed

\begin{remark}
The proof yields convergence of the entire normalized sequence, rather
than only subsequential convergence from a compactness theorem.  The
limit current is normal and has finite mass by the explicit push-forward
construction above.  Its exact total mass is not computed here.  The
current is nonrectifiable, as shown below.
\end{remark}

\begin{remark}[Interpretation]
The current \(\mathcal R_{v^u}\) is the invariant current of the
irrational Kronecker foliation in the unstable direction on the torus.
Taking its product with the conjugacy-sphere current and pushing forward
by \(\Psi\) produces a three-dimensional current on \(\mathcal C\).
On \(\mathcal C_{\mathrm{reg}}\) this is the Ruelle--Sullivan current of
the measured lamination obtained by lifting the Kronecker foliation; the
formula through \(\Psi\) gives its extension across the four central
points \cite{RuelleSullivan}.
\end{remark}

\begin{proposition}[Nonrectifiability of the eigencurrents]
\label{prop:diffuse-rank-two}
On the regular commuting locus $\mathcal C_{\mathrm{reg}}$, each of
$\mathcal T^u$ and $\mathcal T^s$ is represented locally by a smooth
nonvanishing three-vector field against a smooth four-dimensional measure.
In particular its mass measure is absolutely continuous, with positive
smooth density, with respect to four-dimensional volume on
$\mathcal C_{\mathrm{reg}}$.  Consequently $\mathcal T^u$ and
$\mathcal T^s$ are not rectifiable three-currents.
\end{proposition}

\begin{proof}
On $S^2\times\mathbb T^2_*$, the current
$[S^2]\times\mathcal R_{v^u}$ is represented by a smooth nowhere-zero
simple three-vector field against the smooth product measure.  Let
\[
 J(u,\xi)=(-u,-\xi).
\]
The antipodal map reverses the orientation of $S^2$, while
$d(-\id)$ sends the Kronecker vector $v^u$ to $-v^u$.  The two signs cancel,
so the three-current is invariant under $J$.  The quotient map
\[
 S^2\times\mathbb T^2_*\longrightarrow
 (S^2\times\mathbb T^2_*)/\langle J\rangle
 \cong\mathcal C_{\mathrm{reg}}
\]
is a two-sheeted covering.  On every evenly covered open set the two local
push-forwards therefore coincide; the factor $1/2$ in
\eqref{eq:unstable-current-definition} cancels the two sheets.  It follows
that $\mathcal T^u$ is represented there by a smooth nowhere-zero
three-vector field against a smooth four-dimensional measure.  Its mass
measure consequently has positive smooth density with respect to any smooth
four-dimensional volume form on $\mathcal C_{\mathrm{reg}}$.

The preimage under $\Psi$ of each central point has zero product measure,
so no additional mass is concentrated at the singular points.  A rectifiable
three-current has mass concentrated on a countably three-rectifiable set,
and every such set has zero four-dimensional measure.  Hence
$\mathcal T^u$ is not rectifiable.  The argument for $\mathcal T^s$ is the
same.
\end{proof}

\subsection{Periodic sphere fibres, measured laminations, and their currents}
\label{subsec:lamination-current-geometry}

The eigencurrents can be described through the sphere decomposition of the
commuting locus.  Recall that over the regular
pillowcase
\[
 \mathcal P^\circ
 =\bigl(\mathbb T^2\setminus\mathbb T^2[2]\bigr)/\{\xi\sim-\xi\}
 \cong S^2\setminus\{p_1,p_2,p_3,p_4\},
\]
the regular commuting locus is an $S^2$-bundle
\[
 S^2\longrightarrow \mathcal C_{\mathrm{reg}}
 \xrightarrow{\ \pi_{\mathcal C}\ }\mathcal P^\circ.
\]
For $x=[\theta,\varphi]\in\mathcal P^\circ$, the fibre is the simultaneous
conjugacy orbit
\[
 \Sigma_x=\pi_{\mathcal C}^{-1}(x)
 =\left\{\bigl(e^{u\theta},e^{u\varphi}\bigr):u\in S^2\right\}
 \cong S^2.
\]
These fibres are precisely the conjugacy spheres discussed earlier.  There is
one such sphere for every regular point $x$ of the pillowcase, so the full
sphere decomposition of $\mathcal C_{\mathrm{reg}}$ is uncountable.  A fibre $\Sigma_x$ is periodic as a set if and only if $x$ is
periodic for the induced pillowcase map $\overline A$.  Since, for every
$n\geq1$, the set $\operatorname{Fix}(\overline A^{\,n})$ is finite, the
periodic sphere fibres form a countable family.  This family is dense in the
commuting locus because periodic points of the hyperbolic toral automorphism,
and hence of its pillowcase quotient, are dense.

If $x=[\xi]$ has period $n$ on the pillowcase, then
$A^n\xi\equiv\xi$ or $A^n\xi\equiv-\xi$ modulo $2\pi\mathbb Z^2$.
Accordingly, $F_A^n$ restricts to either the identity or the antipodal map on
$\Sigma_x$.  In both cases
\[
 F_A(\Sigma_x)=\Sigma_{\overline A x},
 \qquad
 F_A^{2n}|_{\Sigma_x}=\operatorname{id}_{\Sigma_x}.
\]
Thus a periodic pillowcase orbit gives a finite cycle of sphere fibres, and
some return iterate fixes each sphere pointwise.  If one starts with an actual periodic point of the toral map, the
corresponding return is the identity, as in the periodic conjugacy spheres
constructed earlier.  The periodic spheres form a countable dense subfamily, whereas the
eigencurrents below are three-dimensional objects obtained by sweeping the
sphere fibres along one-dimensional stable or unstable directions in the base.

Let $\mathcal F^u$ and $\mathcal F^s$ denote the irrational Kronecker
foliations of $\mathbb T^2$ determined by the eigenvectors $v^u$ and $v^s$.
They descend, away from the four branch points, to measured one-dimensional
laminations of the punctured pillowcase.  Pulling them back by
$\pi_{\mathcal C}$ gives measured three-dimensional laminations on
$\mathcal C_{\mathrm{reg}}$:
\[
 \mathcal L^u:=\pi_{\mathcal C}^{-1}(\mathcal F^u),
 \qquad
 \mathcal L^s:=\pi_{\mathcal C}^{-1}(\mathcal F^s).
\]
Locally, a leaf of either lamination has the form
\[
 S^2\times\mathbb R,
\]
with two tangent directions along a conjugacy sphere and one tangent direction
along a stable or unstable line in the pillowcase.  They define measured laminations on the regular stratum; the associated currents
extend across the four singular central pairs by the push-forward formula
through $\Psi$.  Around a puncture, the antipodal monodromy reverses the
orientation of the $S^2$ fibre and also reverses the lifted one-dimensional
base direction.  Their product therefore preserves the orientation of the
three-dimensional lamination, which is compatible with the global current.

These laminations belong entirely to the four-dimensional regular commuting
locus.  They are not asserted to define invariant subbundles of the ambient
six-dimensional tangent bundle $T(\SU(2)^2)$; their role here is purely
current-theoretic.

Periodic two-spheres are not leaves of these three-dimensional laminations.
A conjugacy sphere is a fibre and therefore a codimension-one slice inside a
stable or unstable leaf.  More precisely, in a
local trivialization
\[
 \mathcal C_{\mathrm{reg}}\simeq S^2\times U,
 \qquad U\subset\mathbb R^2,
\]
choose local coordinates $(u,s)$ on $U$ tangent respectively to the unstable
and stable directions.  Then
\[
 \Sigma_{(u,s)}=S^2\times\{(u,s)\},
\]
whereas the corresponding local leaves are
\[
 \mathcal L^u_s=S^2\times\mathbb R_u\times\{s\},
 \qquad
 \mathcal L^s_u=S^2\times\{u\}\times\mathbb R_s.
\]
Thus a fibre $\Sigma_x$ is locally the transverse intersection of the stable
and unstable three-leaves through it:
\[
 \Sigma_x=\mathcal L^u(x)\cap\mathcal L^s(x).
\]
For a periodic sphere fibre, choose a return iterate $N$ that fixes the
sphere pointwise; one may take $N=n$ in the identity case above and $N=2n$ in
the antipodal case.  Then the derivative is the identity on the two fibre
directions.  In angular coordinates on the base, it multiplies the unstable
eigendirection by $\mu_u^N$ and the stable eigendirection by $\mu_u^{-N}$.  At $p\in\Sigma_x$,
\[
 T_p\mathcal L^u=T_p\Sigma_x\oplus E^u_{\mathrm{base}},
 \qquad
 T_p\mathcal L^s=T_p\Sigma_x\oplus E^s_{\mathrm{base}}.
\]
The periodic spheres therefore describe neutral fibre directions, whereas the
eigencurrents add one expanding or contracting base direction.
Thus the three-current combines two tangent directions along the sphere with
one direction transverse to the sphere and tangent to the corresponding
three-dimensional lamination.

The unstable current
\[
 \mathcal T^u
 =\frac12\,\Psi_\#\bigl([S^2]\times\mathcal R_{v^u}\bigr)
\]
is the Ruelle--Sullivan current of the measured lamination
\(\mathcal L^u\).  It combines the oriented sphere fibre with the unstable
Kronecker direction and its transverse invariant measure.  The corresponding
statement holds for \(\mathcal T^s\).

The iterated coordinate cycle \((F_A^n)_\#[C_i]\) is the sphere family over
the rational closed geodesic of slope \(A^ne_i\).  After normalization these
closed geodesics converge as currents to the unstable Kronecker current, and
the sphere families converge to \(\mathcal T^u\).  Backward iterates give
\(\mathcal T^s\).

\begin{theorem}[Stable eigencurrent]
\label{thm:stable-current}
Let $v^s\in\mathbb R^2$ satisfy
\[
 Av^s=\mu_u^{-1}v^s.
\]
For $i=1,2$, write
\[
 \mu_u^{-n}A^{-n}e_i\longrightarrow d_i v^s.
\]
Then there exists a nonzero closed normal three-current of finite mass $\mathcal T^s$,
supported on \(\mathcal C\), such that
\[
 \mu_u^{-n}(F_A^{-n})_\#[C_i]
 \longrightarrow d_i\mathcal T^s
\]
weakly as de Rham currents.  It satisfies
\[
 (F_A)_\#\mathcal T^s=\mu_u^{-1}\mathcal T^s,
\]
and its homology class is \(v^s\).
\end{theorem}

\begin{proof}
Apply Theorem~\ref{thm:unstable-current} to the hyperbolic matrix
\(A^{-1}\) and to the inverse Nielsen diffeomorphism \(F_A^{-1}\).  The unstable eigenvalue of $A^{-1}$ on the stable eigendirection of $A$ is
$\mu_u$.  This gives the
stated convergence and
\[
 (F_A^{-1})_\#\mathcal T^s=\mu_u\mathcal T^s.
\]
Applying \((F_A)_\#\) to this identity yields
\((F_A)_\#\mathcal T^s=\mu_u^{-1}\mathcal T^s\).
\end{proof}

\begin{corollary}[Stable and unstable scaling of the eigencurrents]
\label{cor:anosov-currents}
The lines \(\mathbb R\mathcal T^u\) and
\(\mathbb R\mathcal T^s\) are invariant under the push-forward action of
\(F_A\) on three-currents, and
\[
 (F_A^n)_\#\mathcal T^u=\mu_u^n\mathcal T^u,
 \qquad
 (F_A^n)_\#\mathcal T^s=\mu_u^{-n}\mathcal T^s.
\]
Equivalently, their projective classes are fixed.  Thus the lift is not
Anosov on \(S^3\times S^3\), but its stable and unstable eigen-currents
have the expected stable/unstable scaling: their magnitudes grow and decay
exponentially at rate $|\mu_u|$.
\end{corollary}

\begin{remark}
The term \emph{invariant current} should therefore be understood
projectively.  The normalized operators
\[
 \mu_u^{-1}(F_A)_\#
 \quad\hbox{and}\quad
 \mu_u(F_A)_\#
\]
fix \(\mathcal T^u\) and \(\mathcal T^s\), respectively; the unnormalized
push-forward expands or contracts them.
\end{remark}

\section{Higher-rank quaternionic word maps}
\label{sec:higher-rank}

Let $n\geq2$.  For rank at least three the kernel of abelianization contains
non-inner automorphisms, and the commuting locus has higher codimension.  We
use the same word-map construction on $\Sp(1)^n$.

Let
\[
 F_n=\langle x_1,\ldots,x_n\rangle,
 \qquad
 G_n=\Sp(1)^n\cong \SU(2)^n.
\]
If $\Phi\in\Aut(F_n)$ and
\[
 \Phi(x_i)=w_i(x_1,\ldots,x_n),\qquad 1\leq i\leq n,
\]
define
\begin{equation}\label{eq:higher-word-map}
 F_\Phi(q_1,\ldots,q_n)
 =\bigl(w_1(q_1,\ldots,q_n),\ldots,
        w_n(q_1,\ldots,q_n)\bigr).
\end{equation}
As before, we use the row exponent matrix
\[
 R(\Phi)=\ab(\Phi)^t\in \mathrm{GL}(n,\Z).
\]
The usual Nielsen generators show that
$\Aut(F_n)\to\mathrm{GL}(n,\Z)$ is onto; see, for example,
\cite{MagnusKarrassSolitar}.  Fried already considered word maps on an arbitrary number of letters in
\cite{Fried}.

\begin{proposition}\label{prop:higher-basic}
For every $\Phi\in\Aut(F_n)$, the map $F_\Phi$ is a real-analytic
diffeomorphism of $G_n$ and
\[
 F_\Phi^{-1}=F_{\Phi^{-1}}.
\]
It preserves product Haar measure.  If $\Omega_n$ is the product of the
standard oriented volume forms on the $n$ factors, then
\[
 F_\Phi^*\Omega_n=(\det R(\Phi))\Omega_n.
\]
In particular, every $A\in\mathrm{GL}(n,\Z)$ has a real-analytic
Haar-preserving quaternionic lift.
\end{proposition}

\begin{proof}
The inverse follows from word evaluation exactly as in
Proposition~\ref{prop:diffeomorphism}.  The elementary Nielsen generators
are permutations of free generators, inversions of one generator, and
shears $x_i\mapsto x_i x_j$.  Their evaluations preserve product Haar
measure.  A shear is orientation preserving, inversion of one $S^3$ factor
reverses orientation, and a transposition of two $3$-dimensional factors has
sign $(-1)^{3\cdot3}=-1$.  These are the same signs as the determinants of
the corresponding elementary matrices.  The result follows by composition.
\end{proof}

\subsection{The higher-rank commuting locus}

Define
\begin{equation}\label{eq:Cn}
 \mathcal C_n
 =\{(q_1,\ldots,q_n)\in\SU(2)^n:
        q_iq_j=q_jq_i\text{ for all }i,j\}.
\end{equation}
Equivalently,
\[
 \mathcal C_n=\Hom(\Z^n,\SU(2)).
\]
Every commuting family in $\SU(2)$ is contained in a common maximal torus.
Consequently
\begin{equation}\label{eq:Cn-union}
 \mathcal C_n=\bigcup_{T\subset\SU(2)\text{ maximal}}T^n.
\end{equation}
Each $T^n$ is invariant under $F_\Phi$, and in angular coordinates the
restriction is precisely the toral automorphism $R(\Phi)$.

Let
\[
 \T^n[2]=\{0,\pi\}^n,
 \qquad
 \T^n_*:=\T^n\setminus\T^n[2],
 \qquad
 \mathcal P_n:=\T^n/\{\xi\sim-\xi\}.
\]
The map
\[
 \Psi_n:S^2\times\T^n\longrightarrow\mathcal C_n,
 \qquad
 \Psi_n(u,\theta_1,\ldots,\theta_n)
 =\bigl(e^{u\theta_1},\ldots,e^{u\theta_n}\bigr)
\]
is onto and satisfies
\[
 \Psi_n(u,\xi)=\Psi_n(-u,-\xi).
\]

\begin{theorem}\label{thm:higher-commuting}
Let $n\geq2$.  The commuting locus $\mathcal C_n$ has dimension $n+2$ and
codimension $2n-2$ in $\SU(2)^n$.  Its regular part is
\begin{equation}\label{eq:higher-regular}
 \mathcal C_{n,\mathrm{reg}}
 \cong
 \frac{S^2\times\T^n_*}{(u,\xi)\sim(-u,-\xi)},
\end{equation}
and is an $S^2$-bundle over
\[
 \mathcal P_n^\circ
 :=\T^n_*/\{\xi\sim-\xi\}.
\]
The complement consists of the $2^n$ central tuples
\[
 (\varepsilon_1,\ldots,\varepsilon_n),
 \qquad \varepsilon_i\in\{\pm1\},
\]
over which the $S^2$-fibre collapses to one point.  For $n\geq3$ the link
of each central point is
\begin{equation}\label{eq:higher-link}
 L_n=\frac{S^2\times S^{n-1}}
          {(u,a)\sim(-u,-a)},
\end{equation}
and $\pi_1(L_n)\cong\Z/2$.  Hence the central tuples are not manifold
points of $\mathcal C_n$.  For $n=2$ the same conclusion was already
obtained, by the mapping-torus argument, in
Proposition~\ref{prop:commutinglocaltopology}.
\end{theorem}

\begin{proof}
If an $n$-tuple is not central, at least one component is noncentral and
determines its maximal torus, or equivalently its imaginary axis, up to
sign.  Thus on $S^2\times\T^n_*$ the only ambiguity in $\Psi_n$ is
$(u,\xi)\sim(-u,-\xi)$, proving \eqref{eq:higher-regular}.  The dimension
is $n+2$, and the codimension in the $3n$-dimensional ambient manifold is
$2n-2$.

Near a central tuple, exponential coordinates identify the commuting locus
with
\[
 V_n=\{(X_1,\ldots,X_n)\in(\R^3)^n:
          X_i\times X_j=0\text{ for all }i,j\}.
\]
As in the proof of Proposition~\ref{prop:commutinglocaltopology}, this
identification is exact, not merely infinitesimal, on the ball of radius
\(\pi\) where \(\exp\) is injective: any noncentral \(\exp X_i\) has
centralizer exactly the maximal torus through it, so all the \(X_j\) must
lie on the same axis as \(X_i\).
Every nonzero point of $V_n$ can be written
\[
 (X_1,\ldots,X_n)=(a_1u,\ldots,a_nu),
 \qquad
 u\in S^2,
 \quad a=(a_1,\ldots,a_n)\in\R^n\setminus\{0\},
\]
with the unique ambiguity $(u,a)\sim(-u,-a)$.  Intersecting with the unit
sphere gives \eqref{eq:higher-link}.  If $n\geq3$, then
$S^2\times S^{n-1}$ is simply connected and is a two-sheeted cover of
$L_n$, so $\pi_1(L_n)\cong\Z/2$ and hence
$H_1(L_n;\mathbb Z)\cong\mathbb Z/2$.  A neighborhood of the central tuple
is the cone on $L_n$, and therefore its local homology in degree $2$ is
\[
 H_2(\mathcal C_n,\mathcal C_n\setminus\{p\};\mathbb Z)
 \cong\widetilde H_1(L_n;\mathbb Z)\cong\mathbb Z/2.
\]
Since $n+2\geq5$, the local homology of an $(n+2)$-manifold vanishes in
degree $2$.  Hence the central tuple is singular.
\end{proof}

\subsection{Entropy and periodic conjugacy spheres}

Assume now that $A=R(\Phi)\in\mathrm{GL}(n,\Z)$ is hyperbolic, meaning
that no eigenvalue of $A$ has modulus one.  The restriction to every
invariant $T^n$ is the hyperbolic toral automorphism $A$.  Therefore
\begin{equation}\label{eq:higher-entropy-lower}
 h_{\mathrm{top}}(F_\Phi)
 \geq h_{\mathrm{top}}(A)
 =\sum_{|\lambda_j|>1}\log|\lambda_j|.
\end{equation}
For $n=2$, Fried's theorem gives equality.  No corresponding equality is
proved here for $n\geq3$.

The simultaneous conjugation action of $\SU(2)$ on $G_n$ commutes with
every word map.  If $p\in T^n$ is a noncentral periodic point of least
period $m$, then its conjugacy orbit
\[
 \mathcal O_p
 =\{(hq_1h^{-1},\ldots,hq_nh^{-1}):h\in\SU(2)\}
 \cong \SU(2)/S^1\cong S^2
\]
is fixed pointwise by $F_\Phi^m$.

\begin{corollary}\label{cor:higher-not-anosov}
For every $n\geq2$, a quaternionic lift of a hyperbolic
$A\in\mathrm{GL}(n,\Z)$ is not an Anosov diffeomorphism of $\SU(2)^n$.
The periodic conjugacy two-spheres form a countable dense family in
$\mathcal C_n$.
\end{corollary}

\begin{proof}
Periodic points of a hyperbolic toral automorphism are countable and dense
in $\T^n$.  Their conjugacy orbits therefore give a countable dense family
of sphere fibres in $\mathcal C_n$.  At a noncentral periodic point the
return derivative is the identity on the two tangent directions to the
conjugacy sphere, and hence the periodic point is not hyperbolic.
\end{proof}

Consider the three-generator automorphism
\begin{equation}\label{eq:rank-three-example}
 \sigma(x_1)=x_2,
 \qquad
 \sigma(x_2)=x_3,
 \qquad
 \sigma(x_3)=x_1x_2.
\end{equation}
This map also occurs in the unpublished notes of A. Gogolev and
F. Rodr\'iguez Hertz, where its ergodicity is posed as a question.  Then
\[
 F_\sigma(q_1,q_2,q_3)=(q_2,q_3,q_1q_2),
 \qquad
 F_\sigma^{-1}(u,v,w)=(wu^{-1},u,v),
\]
and
\[
 R(\sigma)=
 \begin{pmatrix}
 0&1&0\\
 0&0&1\\
 1&1&0
 \end{pmatrix}.
\]
Its characteristic polynomial is $t^3-t-1$.  Hence it is hyperbolic: it
has one real eigenvalue greater than one and a complex-conjugate pair of
modulus less than one.  The only fixed point of $F_\sigma$ is $(1,1,1)$,
although noncentral periodic points occur on every invariant maximal torus.
The existence of the periodic spheres therefore does not by itself decide
ergodicity with respect to Haar measure, since their union has Haar measure
zero in the ambient manifold.

\subsection{IA-automorphisms in rank at least three}

For rank two, $IA_2=\operatorname{Inn}(F_2)$, and the reduced map is determined by the abelianization matrix.  For $n\geq3$, the following IA-automorphism acts nontrivially on the character variety.

\begin{proposition}\label{prop:IA-character}
Let $n\geq3$.  The automorphism
\begin{equation}\label{eq:tau-IA}
 \tau(x_1)=x_1[x_2,x_3],
 \qquad
 \tau(x_j)=x_j\quad (j\geq2)
\end{equation}
belongs to $IA_n=\ker(\Aut(F_n)\to\mathrm{GL}(n,\Z))$, but its action on
$\mathcal X_{\SU(2)}(F_n)$ is nontrivial.  Consequently, for $n\geq3$ a
matrix $A\in\mathrm{GL}(n,\Z)$ does not in general determine a canonical
map on the $\SU(2)$-character variety.
\end{proposition}

\begin{proof}
The inverse of $\tau$ is obtained by replacing $[x_2,x_3]$ with its
inverse, so $\tau$ is an automorphism.  Since a commutator has zero
exponent sum, $\tau$ acts trivially on abelianization.

To see that its character action is nontrivial, take
\[
 q_1=1,\qquad q_2=\mathbf i,\qquad q_3=\mathbf j,
\]
and take the remaining coordinates equal to $1$ if $n>3$.  Then
$[\mathbf i,\mathbf j]=-1$, so evaluation of \eqref{eq:tau-IA} changes the
first coordinate from $1$ to $-1$.  The trace of the first coordinate
therefore changes from $2$ to $-2$, and the two representations cannot be
simultaneously conjugate.

If $F_\Phi$ is any lift of $A$, then $\tau\circ\Phi$ has the same
abelianization as $\Phi$.  Since the reduced action of $\tau$ is
nontrivial, the two induced maps on the character variety need not agree.
\end{proof}

\begin{remark}
All lifts with the same matrix $A$ nevertheless agree on $\mathcal C_n$.
Indeed, on a commuting tuple every word reduces to its exponent sums.
Thus the commuting dynamics is determined by $A$, whereas for $n\geq3$
the dynamics away from $\mathcal C_n$ can retain information from the
chosen element of $\Aut(F_n)$.
\end{remark}

\subsection{Dependence on the choice of lift}
\label{subsec:dependence-lift}

If $R(\Phi)=R(\Psi)=A$, then $\Psi\Phi^{-1}\in IA_n$.

\begin{proposition}\label{prop:matrix-versus-IA}
Let $\Phi,\Psi\in\Aut(F_n)$ have the same row exponent matrix
$A\in\mathrm{GL}(n,\Z)$.  Then:
\begin{enumerate}
\item $F_\Phi$ and $F_\Psi$ have the same restriction to every common
maximal torus $T^n$ and hence to the whole commuting locus $\mathcal C_n$;
\item they induce the same homomorphism on every homology group of
$(S^3)^n$, namely $\Lambda^kA$ on $H_{3k}$;
\item if $n=2$, they induce the same map on the $\SU(2)$-character variety;
\item if $n\geq3$, the conclusion in (3) can fail.
\end{enumerate}
\end{proposition}

\begin{proof}
The first assertion follows because word evaluation on commuting tuples
depends only on exponent sums.  The second is
Theorem~\ref{thm:exterior-homology}.  The third is
Proposition~\ref{prop:nonunique}, using $IA_2=\operatorname{Inn}(F_2)$, and the fourth is
Proposition~\ref{prop:IA-character}.
\end{proof}

Periodic data contained in $\mathcal C_n$ and the homological spectrum are
determined by the matrix.  Periodic orbits outside $\mathcal C_n$, metric
properties, and higher-dimensional iterated cycles are not determined by
these arguments and may depend on the chosen free-group automorphism.
Whether entropy or ergodicity varies inside a fixed abelianization class
remains open.

\subsection{Exterior powers in homology}

Let $C_i\subset(S^3)^n$ be the $i$-th coordinate three-sphere.  Their
classes form the standard basis of
\[
 H_3((S^3)^n;\Z)\cong\Z^n.
\]
The K\"unneth theorem gives natural bases of the groups
$H_{3k}((S^3)^n;\mathbb Z)$ indexed by $k$-element subsets of
$\{1,\ldots,n\}$.  Dually, the cohomology ring is the exterior algebra on
the degree-three generators.

\begin{theorem}\label{thm:exterior-homology}
Let $A=R(\Phi)$.  For $0\leq k\leq n$,
\begin{equation}\label{eq:exterior-action}
 (F_\Phi)_*|_{H_{3k}((S^3)^n;\Z)}=\Lambda^k A
\end{equation}
under the natural identification
\[
 H_{3k}((S^3)^n;\Z)\cong\Lambda^k\Z^n.
\]
In particular, the eigenvalues on $H_{3k}((S^3)^n;\R)$ are the products
$\lambda_{i_1}\cdots\lambda_{i_k}$ of $k$ eigenvalues of $A$.
\end{theorem}

\begin{proof}
Setting all but one free generator equal to the identity shows, exactly as
in Section~11, that the action on $H_3$ is $A$.  Let
$\alpha_1,\ldots,\alpha_n$ be the basis of $H^3((S^3)^n;\mathbb Z)$ dual
to the coordinate three-spheres.  The induced map on $H^3$ is therefore
$A^t$.  Since
\[
 H^*((S^3)^n;\mathbb Z)
 \cong \Lambda^*(\alpha_1,\ldots,\alpha_n)
\]
as a graded ring and $F_\Phi^*$ is a ring homomorphism, its restriction to
$H^{3k}$ is $\Lambda^k(A^t)$.  Passing to the Kronecker-dual basis of
$H_{3k}$ gives
\[
 (F_\Phi)_*|_{H_{3k}}=(\Lambda^k(A^t))^t=\Lambda^kA.
\]
\end{proof}

For $n=3$ this gives
\[
 H_3((S^3)^3;\Z)\cong\Z^3,
 \qquad
 H_6((S^3)^3;\Z)\cong\Z^3,
 \qquad
 H_9((S^3)^3;\Z)\cong\Z,
\]
with induced actions $A$, $\Lambda^2A$, and $\det A$, respectively.

\subsection{A higher-rank eigencurrent}

The current construction of Section~11 also has a higher-rank form.  For
$v=(v_1,\ldots,v_n)\in\Z^n$, put
\[
 \phi_v:S^3\longrightarrow(S^3)^n,
 \qquad
 \phi_v(q)=(q^{v_1},\ldots,q^{v_n}),
\]
and let
\[
 \gamma_v:S^1\longrightarrow\T^n,
 \qquad
 \gamma_v(t)=tv\pmod{2\pi}.
\]
With $E:S^2\times S^1\to S^3$ as in Section~11, one has
\[
 \phi_v\circ E
 =\Psi_n\circ(\id_{S^2}\times\gamma_v),
\]
and hence
\begin{equation}\label{eq:higher-cycle-current}
 (\phi_v)_\#[S^3]
 =\frac12(\Psi_n)_\#\bigl([S^2]\times[\gamma_v]\bigr).
\end{equation}

For $w\in\R^n$, define the constant one-current on $\T^n$ by
\[
 \mathcal R_w(\alpha)
 =\frac1{(2\pi)^{n-1}}\int_{\T^n}\alpha_x(w)\,dx.
\]
A vector $w$ will be called \emph{totally irrational} if
\[
 k\cdot w\neq0
 \qquad\text{for every }0\neq k\in\Z^n.
\]

\begin{lemma}\label{lem:higher-fourier}
Let $v_j\in\Z^n$ be primitive and let $s_j\in\R\setminus\{0\}$ satisfy
\[
 s_j^{-1}v_j\longrightarrow w,
 \qquad |s_j|\longrightarrow\infty,
\]
where $w$ is totally irrational.  Then
\[
 s_j^{-1}[\gamma_{v_j}]\longrightarrow\mathcal R_w
\]
weakly as one-currents on $\T^n$.
\end{lemma}

\begin{proof}
Let $\alpha$ be a smooth one-form on $\T^n$ and expand its coefficients in
Fourier series.  As in Lemma~\ref{lem:rational-geodesics}, a nonzero Fourier
mode $k\in\Z^n$ contributes only when $k\cdot v_j=0$.  The zero mode gives
precisely the required limit.  It remains to control the resonant modes.

Let
\[
 m(v_j)=\min\{|k|:0\neq k\in\Z^n,\ k\cdot v_j=0\}.
\]
Then $m(v_j)\to\infty$.  Otherwise a bounded sequence of resonant vectors
would have a constant nonzero subsequence $k$, and division by $s_j$ would
give $k\cdot w=0$, contrary to total irrationality.  Let $N>n+1$.  Smoothness gives
$|\widehat\alpha(k)|\leq C_N(1+|k|)^{-N}$.  The nonzero resonant part,
after division by $s_j$, is therefore bounded by a constant times
\[
 \frac{|v_j|}{|s_j|}
 \sum_{|k|\geq m(v_j)}(1+|k|)^{-N}.
\]
The first factor is bounded because $v_j/s_j\to w$, while the sum tends to
zero as $m(v_j)\to\infty$.  This proves the weak convergence.
\end{proof}

\begin{theorem}\label{thm:higher-eigencurrent}
Let $A=R(\Phi)\in\mathrm{GL}(n,\Z)$ have a real simple eigenvalue $\mu$
which is strictly dominant in modulus, and let $w\in\R^n$ be a totally
irrational eigenvector,
\[
 Aw=\mu w,
 \qquad
 |\mu|>|\lambda|
 \quad\text{for every other eigenvalue }\lambda\text{ of }A.
\]
For each coordinate vector $e_i$, write
\[
 \mu^{-m}A^m e_i\longrightarrow c_iw.
\]
At least one coefficient $c_i$ is nonzero.  Then the current
\begin{equation}\label{eq:higher-current-def}
 \mathcal T_w
 :=\frac12(\Psi_n)_\#\bigl([S^2]\times\mathcal R_w\bigr)
\end{equation}
is a nonzero closed normal three-current of finite mass and satisfies
\[
 (F_\Phi)_\#\mathcal T_w=\mu\mathcal T_w.
\]
Moreover,
\[
 \mu^{-m}(F_\Phi^m)_\#[C_i]
 \longrightarrow c_i\mathcal T_w
\]
weakly as three-currents, and
\[
 [\mathcal T_w]=w
 \quad\text{in}\quad H_3((S^3)^n;\R)\cong\R^n.
\]
\end{theorem}

\begin{proof}
Since $\mu$ is a simple eigenvalue separated in modulus from the rest of
the spectrum, there is a real left eigenvector $\ell$ with
$\ell(w)=1$ such that
\[
 \mu^{-m}A^m\longrightarrow w\otimes\ell.
\]
Thus $c_i=\ell(e_i)$, and at least one $c_i$ is nonzero.  Because $A^m$ is
unimodular, each $A^m e_i$ is primitive.  If $c_i\neq0$, then $c_iw$ is
totally irrational, and Lemma~\ref{lem:higher-fourier}, applied to
$v_m=A^m e_i$ and $s_m=\mu^m$, gives
\[
 \mu^{-m}[\gamma_{A^m e_i}]\longrightarrow c_i\mathcal R_w.
\]
If $c_i=0$, the spectral gap gives
$\|A^m e_i\|=o(|\mu|^m)$.  Since the mass of
$[\gamma_v]$ is $2\pi\|v\|$ for the flat metric, the mass of
$\mu^{-m}[\gamma_{A^m e_i}]$ tends to zero.  In either case,
\eqref{eq:higher-cycle-current} gives the asserted convergence.  The
current $\mathcal R_w$ is closed and has finite mass; therefore so does its
product with $[S^2]$, and smooth push-forward by $\Psi_n$ preserves these
properties.  Finally,
\[
 F_\Phi\circ\Psi_n
 =\Psi_n\circ(\id_{S^2}\times A)
\]
on the commuting locus, and hence
\[
 (F_\Phi)_\#\mathcal T_w
 =\frac12(\Psi_n)_\#\bigl([S^2]\times\mathcal R_{Aw}\bigr)
 =\mu\mathcal T_w.
\]
Finally, the homology class of $(F_\Phi^m)_\#[C_i]$ is $A^m e_i$.
At least one coefficient $c_i$ is nonzero.  For such an $i$, pairing the
weak limit with closed three-forms gives
$c_i[\mathcal T_w]=c_iw$, and hence $[\mathcal T_w]=w$.  In particular
$\mathcal T_w\neq0$.
\end{proof}

\begin{corollary}\label{cor:higher-diffuse}
For every $n\geq2$ under the hypotheses of
Theorem~\ref{thm:higher-eigencurrent}, the current $\mathcal T_w$ is
nonrectifiable.  More precisely,
on $\mathcal C_{n,\mathrm{reg}}$ its mass measure is absolutely continuous,
with positive smooth density, with respect to $(n+2)$-dimensional volume.
In particular $\mathcal T_w$ is not a rectifiable three-current.
\end{corollary}

\begin{proof}
On $S^2\times\mathbb T^n_*$ the current
$[S^2]\times\mathcal R_w$ is represented by a smooth nowhere-zero simple
three-vector field against the smooth $(n+2)$-dimensional product measure.
It is invariant as a current under
$J(u,\xi)=(-u,-\xi)$: the antipodal map reverses the oriented sphere
factor and $d(-\id)$ reverses the Kronecker vector, so the signs cancel.
Thus, on an evenly covered open set of $\mathcal C_{n,\mathrm{reg}}$, the
two local push-forwards agree and the factor $1/2$ in
\eqref{eq:higher-current-def} removes the double covering.  Hence the mass
measure of $\mathcal T_w$ has positive smooth density with respect to
$(n+2)$-dimensional volume on the regular locus.  The central tuples receive
no mass because their preimages have zero product measure.  Since $n+2>3$,
a countably three-rectifiable set has zero $(n+2)$-dimensional measure, so
$\mathcal T_w$ cannot be rectifiable.
\end{proof}

\begin{remark}
The total irrationality hypothesis is automatic in rank two for an
eigenvector of a hyperbolic integral matrix, but not in higher rank.  If an
eigenvector lies in a proper rational subspace, the corresponding rational
geodesics need not become equidistributed in the whole torus.  In that case
the appropriate limiting current should be formulated on the smallest
invariant subtorus containing the eigenline.
\end{remark}

\subsection{Higher-dimensional eigenclasses and Ruelle--Sullivan cycles}

Theorem~\ref{thm:higher-eigencurrent} realizes the degree-three
eigendirection geometrically.  In higher homological degrees the exterior-power
formula of Theorem~\ref{thm:exterior-homology} gives further eigendirections,
but their geometric realization is different from the degree-three case.

For an ordered subset $I=\{i_1<\cdots<i_k\}\subset\{1,\ldots,n\}$ let
\[
 C_I=\{1\}\times\cdots\times S^3_{i_1}\times\cdots\times
 S^3_{i_k}\times\cdots\times\{1\}
\]
be the corresponding coordinate $3k$-cycle.  Its homology class is the basis
vector $e_I=e_{i_1}\wedge\cdots\wedge e_{i_k}$ of
$\Lambda^k\Z^n$.  Hence
\[
 [(F_\Phi^m)_\#[C_I]]=(\Lambda^kA)^m e_I.
\]
Thus every eigenvalue of $\Lambda^kA$ governs a possible exponential
homological growth rate for iterated $3k$-cycles.

\begin{proposition}[A compactness criterion for higher eigencycles]
\label{prop:higher-current-compactness}
Let $\mu\in\R$ be a simple eigenvalue of $\Lambda^kA$, strictly dominant in
modulus on the cyclic subspace generated by $e_I$, and suppose
\[
 \mu^{-m}(\Lambda^kA)^m e_I\longrightarrow c\,\xi,
 \qquad (\Lambda^kA)\xi=\mu\xi,
 \qquad c\neq0.
\]
Assume in addition the geometric mass estimate
\begin{equation}\label{eq:higher-mass-bound}
 \mathbf M\bigl((F_\Phi^m)_\#[C_I]\bigr)
 \leq C|\mu|^m
 \qquad(m\geq0).
\end{equation}
Then the normalized currents
\[
 T_m:=\mu^{-m}(F_\Phi^m)_\#[C_I]
\]
have uniformly bounded mass and zero boundary.  Consequently they are
precompact in the weak topology of normal currents.  Moreover the Ces\`aro
averages
\[
 S_N:=\frac1N\sum_{m=0}^{N-1}T_m
\]
have weakly convergent subsequences, and every such limit $S$ is a nonzero
closed normal $3k$-current satisfying
\[
 (F_\Phi)_\#S=\mu S,
 \qquad [S]=c\,\xi.
\]
\end{proposition}

\begin{proof}
The mass estimate gives $\sup_m\mathbf M(T_m)<\infty$, while
$\partial T_m=0$.  Since the ambient manifold is compact, the compactness
theorem for normal currents gives weak precompactness; see
\cite{Federer}.  The same bound applies to $S_N$.  Since
\[
 (F_\Phi)_\#T_m=\mu T_{m+1},
\]
we have
\[
 (F_\Phi)_\#S_N-\mu S_N
 =\frac{\mu}{N}(T_N-T_0),
\]
which tends to zero weakly because the masses of $T_N$ are uniformly bounded.
Every weak limit therefore satisfies $(F_\Phi)_\#S=\mu S$.  The asserted
homology class follows from the convergence of
$\mu^{-m}(\Lambda^kA)^m e_I$ and averaging.  Since $c\xi\neq0$, the limit
current is nonzero.
\end{proof}

For $k\geq2$, the additional analytic hypothesis is the estimate
\eqref{eq:higher-mass-bound}.  For $k=1$ the exact parametrization through a
rational geodesic in $\T^n$ gives the required control.  For $k\geq2$, after
restricting to $C_I$ several quaternionic variables remain, and the iterated
word maps need not commute.  The mass growth therefore depends on the
noncommutative differential of the word map and is not determined by
$\Lambda^kA$ alone.

In degree three, the current $\mathcal T_w$ is a Ruelle--Sullivan current.  Indeed, the one-dimensional Kronecker foliation of
$\T^n$ in the direction $w$ descends away from the $2$-torsion points and its
pullback by the conjugacy-sphere bundle has local leaves
\[
 S^2\times\R.
\]
The transverse invariant measure of the Kronecker foliation, together with
the oriented sphere fibre, gives precisely the current
$\mathcal T_w$ of \eqref{eq:higher-current-def}.

For $k\geq2$ an ordinary pullback of a $k$-dimensional linear foliation of
$\T^n$ would have leaves of dimension $k+2$, not $3k$.  Hence it cannot
represent a general class in $H_{3k}((S^3)^n)$.  If limits as in Proposition~\ref{prop:higher-current-compactness} exist,
Such limits belong to the setting of Sullivan's generalized cycles.  The iterated coordinate manifolds may fold
through the noncommuting region and converge as currents without converging
to leaves of a smooth foliation; see \cite{RuelleSullivan,SullivanCycles}.

Tangent-plane information can be retained by lifting an immersed
$3k$-cycle to the bundle of oriented unit simple $3k$-vectors in
$T((S^3)^n)$.  Give the lifted cycle the orientation multiplied by
$\operatorname{sgn}(\mu)^m$ and let $\nu_m$ be its mass measure normalized
by $|\mu|^{-m}$.  Then
\[
 T_m(\omega)=\int \langle\omega_x,\tau\rangle\,d\nu_m(x,\tau).
\]
Under the bound \eqref{eq:higher-mass-bound}, the measures $\nu_m$ have
weak-* accumulation points.  Every such limit defines a current by the same
formula.  Concentration on a measurable plane field is an additional
property; integrability of that field and holonomy invariance of a transverse
measure are further conditions needed to obtain a measured lamination and a
Ruelle--Sullivan current.

Since
$\dim\mathcal C_n=n+2$, if $3k>n+2$ then
$\mathcal H^{3k}(\mathcal C_n)=0$.  Consequently no nonzero rectifiable
$3k$-current can be supported on $\mathcal C_n$.  Thus any rectifiable
higher-dimensional eigen-current in this range must meet the noncommuting
region.  A weak limit supplied only by compactness of normal currents
need not be rectifiable.

For $n=3$ the first new degree is $6$.  The commuting locus has dimension
five, so a nonzero rectifiable six-current cannot be carried by it.  At the
same time, for $A\in\mathrm{SL}(3,\Z)$ there is the canonical duality
\begin{equation}\label{eq:wedge2-duality}
 \Lambda^2A\simeq(A^{-1})^{t},
\end{equation}
obtained by contracting with the oriented volume element of $\R^3$, i.e.\
by the cyclic identification $\Lambda^2\R^3\cong\R^3$,
$e_2\wedge e_3\leftrightarrow e_1$, $e_3\wedge e_1\leftrightarrow e_2$,
$e_1\wedge e_2\leftrightarrow e_3$, which is the identification for which
the stated formula holds.  Thus
the degree-six eigenvalues are the reciprocals of the degree-three
eigenvalues.  Expansion in $H_6$ is therefore dual to contraction in $H_3$.
Thus a six-dimensional current representing the expanding homology class on
$(S^3)^3$ cannot be accounted for by the commuting-locus geometry alone.

\begin{question}\label{q:higher-current}
For $k\geq2$, for which Nielsen word maps and coordinate cycles does the
mass estimate \eqref{eq:higher-mass-bound} hold?  When it holds, can the
resulting eigen-current be represented by a Ruelle--Sullivan current carried
by a measured lamination, or does it require a more general closed current?  In particular, what is the geometric realization of the
expanding eigendirection of $\Lambda^2A$ for a hyperbolic
$A\in\mathrm{SL}(3,\Z)$?
\end{question}

The higher-degree problem requires the mass estimate \eqref{eq:higher-mass-bound}.

\begin{question}[Higher-dimensional eigen-current problem]
\label{prob:higher-eigencurrent}
Let $\xi\in\Lambda^k\R^n$ be a real eigendirection of $\Lambda^kA$ with
eigenvalue $\mu$, $|\mu|>1$.  Determine conditions under which a finite
linear combination $C$ of coordinate $3k$-cycles, with homology having a
nonzero component in the direction $\xi$, has normalized iterates or
normalized Ces\`aro averages converging to a nonzero closed normal current
$\mathcal T_\xi$ satisfying
\[
 (F_\Phi)_\#\mathcal T_\xi=\mu\mathcal T_\xi,
 \qquad [\mathcal T_\xi]=\xi
\]
after normalization of $\xi$.  When $3k>n+2$, determine how the support and
tangent-plane measure of $\mathcal T_\xi$ enter the noncommuting region.
\end{question}

For $k=1$ this problem is solved above under the stated dominance and
irrationality hypotheses, and the limiting object is the Ruelle--Sullivan
current of the lifted Kronecker measured lamination.  For $k\geq2$ the
problem is different: neither the exterior-power homology formula
nor the geometry of the commuting locus supplies the needed mass control.

Two further questions arise in rank at least three.

\begin{question}\label{q:higher-entropy}
For $n\geq3$, when does equality hold in
\eqref{eq:higher-entropy-lower}?  In particular, can two lifts with the same
abelianization matrix have different topological entropy?
\end{question}

\begin{question}\label{q:higher-ergodicity}
Can a hyperbolic word map on $\Sp(1)^n$, $n\geq3$, be ergodic with respect
to product Haar measure?  In particular, what are the ergodic properties of
the rank-three map \eqref{eq:rank-three-example}?
\end{question}

Goldman proved that for $n>2$ the group $\mathrm{Out}(F_n)$ acts
ergodically on $\operatorname{Hom}(F_n,\SU(2))/\SU(2)$ with respect to
the Lebesgue measure class \cite{GoldmanOutFn}.  This theorem concerns the
full group action and does not decide ergodicity of a single hyperbolic word
map.

The invariant tori and periodic conjugacy spheres lie in subsets of zero ambient Haar measure for $n\geq2$.  Proposition~\ref{prop:IA-character} shows that, in rank at least three, the action away from the commuting locus can depend on data not visible in abelianization.

\section{Closures of word diffeomorphisms}
\label{sec:closures-word-diffeomorphisms}

Let
\[
 \mathcal N=\{F_\Phi:\Phi\in\Aut(F_2)\}
 \subset \operatorname{Diff}^{\omega}(S^3\times S^3)
\]
be the group of word diffeomorphisms considered above.  One may ask for
its closure in the $C^0$, $C^1$, and $C^\infty$ topologies.

The action on
\[
 H_3(S^3\times S^3;\mathbb Z)\cong\mathbb Z^2
\]
is locally constant in the $C^0$ topology.  Indeed, sufficiently close maps
from a compact Riemannian manifold to itself are homotopic by joining
corresponding points with the unique short geodesic.  It follows that if a
sequence $F_{\Phi_j}$ converges in the $C^0$ topology, then the matrices
$R(\Phi_j)\in\mathrm{GL}(2,\mathbb Z)$ are eventually constant.

In particular, accumulation at the identity is governed by
\[
 IA_2=\ker\bigl(\Aut(F_2)\longrightarrow\mathrm{GL}(2,\mathbb Z)\bigr)
      =\operatorname{Inn}(F_2).
\]
If $\iota_w$ denotes conjugation by $w\in F_2$ and
$W=w(q_1,q_2)$, then
\begin{equation}\label{eq:inner-closure-map}
 F_{\iota_w}(q_1,q_2)
 =\bigl(Wq_1W^{-1},Wq_2W^{-1}\bigr).
\end{equation}
Every map \eqref{eq:inner-closure-map} fixes the commuting locus
$\mathcal C$ pointwise.  It also induces the identity on the
$\mathrm{SU}(2)$-character variety.

A theorem of Thom gives a sequence of nontrivial words $w_j\in F_2$ such
that, for every compact group $G$ and every neighborhood $V$ of the
identity, one has
\[
 w_j(G\times G)\subset V
\]
for all sufficiently large $j$; see \cite{Thom2017}.  Applied to
$G=\mathrm{SU}(2)$ and to a bi-invariant metric $d$, this gives
\[
 \sup_{q_1,q_2\in\mathrm{SU}(2)}
 d\bigl(w_j(q_1,q_2),1\bigr)\longrightarrow0.
\]
Equation \eqref{eq:inner-closure-map} therefore implies
\[
 F_{\iota_{w_j}}\longrightarrow\operatorname{id}
\]
uniformly on $S^3\times S^3$.  The inverses converge uniformly as well,
since $w_j^{-1}$ has the same property.  The maps $F_{\iota_{w_j}}$ are
nontrivial.  Indeed, if $F_{\iota_w}$ were the identity, then
$w(q_1,q_2)$ would commute with both entries for every pair.  For the open
dense set of pairs whose common centralizer is $\{\pm1\}$, the value of
$w$ would therefore lie in $\{\pm1\}$.  By density and continuity the
whole image of the word map would lie in $\{\pm1\}$; since
$\mathrm{SU}(2)^2$ is connected and $w(1,1)=1$, the word map would be
identically $1$.  This is impossible for a nontrivial word because
$\mathrm{SU}(2)$ contains nonabelian free subgroups.  Thus
\begin{equation}\label{eq:C0-nondiscrete}
 \mathcal N\quad\text{is not discrete in the $C^0$ topology.}
\end{equation}

The same argument does not give convergence in the $C^1$ topology, since
uniform control of a word map does not control its derivatives.  The
corresponding closure problems are the following.

\begin{question}\label{q:C0-closure}
What is the $C^0$ closure of $\mathcal N$ in
$\operatorname{Homeo}(S^3\times S^3)$?
\end{question}

\begin{question}\label{q:C1-closure}
Is $\mathcal N$ discrete in the $C^1$ topology?  What are its closures in
the $C^1$ and $C^\infty$ topologies?
\end{question}

\section*{Acknowledgments}

The author acknowledges {\bf Proyecto PAPIIT IN103324
(DGAPA, UNAM, M\'exico)} for its financial support.
The author thanks Giovanni Forni for suggesting a broader treatment of the
subject and for putting him in contact with Andrey Gogolev and Federico
Rodr\'iguez Hertz.  He thanks Gogolev and Rodr\'iguez Hertz for bringing
Fried's paper \cite{Fried} to his attention and for sharing unpublished notes
and references on word maps, mapping classes, entropy, and ergodicity.

\medskip
\noindent\textbf{Use of Generative-AI tools declaration.}
The author used ChatGPT (OpenAI) and Claude (Anthropic) during the preparation
of this manuscript for proofreading, checking calculations and mathematical
arguments, and improving the clarity and exposition of the text.  All
AI-assisted suggestions were reviewed and, where appropriate, independently
verified by the author.  The author assumes full responsibility for the
mathematical content and for the final version of the manuscript.

\end{document}